\documentclass[reqno]{amsart}
\usepackage[utf8]{inputenc}
\usepackage{graphicx}
\usepackage{mathtools}
\usepackage{xcolor}
\usepackage[parfill]{parskip}
\usepackage{tikz}
\usetikzlibrary{arrows, backgrounds, calc, chains, decorations, patterns, positioning, shapes, decorations.pathreplacing}
\usepackage{amsfonts}
\usepackage{amsmath}
\usepackage[alphabetic]{amsrefs}
\usepackage{amssymb}
\usepackage{amstext}
\usepackage{amsthm}
\usepackage{dsfont}
\usepackage{todonotes}

\usepackage{caption}
\usepackage{enumitem}
\usepackage[centering]{geometry}
\usepackage{hyperref}
\usepackage{cleveref}
\usepackage{subcaption}
\usepackage{faktor}
\DeclareMathOperator{\Span}{Span}
\DeclareMathOperator{\Aut}{Aut}
\DeclareMathOperator{\ch}{ch}
\DeclareMathOperator{\CH}{CH}
\DeclareMathOperator{\Hom}{Hom}
\DeclareMathOperator{\Ind}{Ind}
\DeclareMathOperator{\Lie}{Lie}
\DeclareMathOperator{\Exp}{Exp}
\DeclareMathOperator{\OS}{OS}

\DeclareMathOperator{\Res}{Res}
\DeclareMathOperator{\rk}{rk}
\DeclareMathOperator{\sgn}{sgn}
\DeclareMathOperator{\Specht}{Specht}
\DeclareMathOperator{\tr}{Tr}
\DeclareMathOperator{\Vrep}{VRep}

\DeclareMathOperator{\supp}{supp}
\DeclareMathOperator{\Logg}{Log}
\newcommand{\C}{\mathbb{C}}

\newcommand{\G}{\mathcal{G}}
\newcommand{\LL}{\mathcal{L}}
\newcommand{\N}{\mathbb{N}}

\newcommand{\Z}{\mathbb{Z}}
\newcommand{\even}{\textnormal{even}}
\newcommand{\odd}{\textnormal{odd}}
\newcommand{\Q}{\mathbb{Q}}
\newcommand{\F}{\mathbb{F}}

\theoremstyle{plain}
\newtheorem{theorem}{Theorem}[section]

\newtheorem{lemma}[theorem]{Lemma}

\theoremstyle{definition}
\newtheorem{definition}[theorem]{Definition}

\theoremstyle{remark}
\newtheorem{remark}[theorem]{Remark}
\newtheorem{example}[theorem]{Example}

\title{Equivariant Chow functions and Wonderful Models \\ for classical root systems}
\author{Roberto Pagaria, Tommaso Scognamiglio}
\date{\today}

\begin{document}

\begin{abstract}
We prove a general inversion formula in the equivariant incidence algebras of geometric lattices that relates the equivariant Chow function with the equivariant super-reduced characteristic function.
We apply this result to the case of minimal wonderful models for hyperplane arrangements associated to root systems, with the natural action of the corresponding Weyl group.
Applying this formula to type $A$ (i.e.\ the partition lattices), we recover Getzler's compositional identity in type $A$.
In type $B$ (i.e.\ the signed partition lattice), we recover Henderson's result, expressing the equivariant Chow generating series in terms of the same series in type $A$ and the equivariant characteristic generating series of type $B$.
Moreover, we obtain a new compositional identity for the equivariant Chow generating series. In type $D$, we obtain a formula expressing the equivariant generating Chow series in terms of the same series in type $A$ and the equivariant characteristic generating series of type $B$.
\end{abstract}

\maketitle

\noindent \textbf{Keywords:} Hyperplane arrangements, Wonderful models, Lattices, Cohomology, Group actions, Weyl groups.

\noindent \textbf{Mathematics Subject Classification:} 05E05, 14N20, 14H10.
\section{Introduction}

In this paper, we study finite group actions on hyperplane arrangements, through combinatorial techniques coming from matroid theory. We start by giving a brief historical overview of the subject.

The study of hyperplane arrangements is a classical field in topology, algebraic geometry and combinatorics, whose interest originally came from the study of braid groups and configuration spaces. More precisely, for an arrangement $\mathcal{A}$ of hyperplanes of a vector space $V$, we are interested in the study of the complement $\displaystyle M(\mathcal{A})=V \setminus \bigcup \nolimits_{H \in \mathcal{A}} H$ and its projectivization $\mathbb{P}(M(\mathcal{A}))$.
A pioneering step in the understanding of these objects came from the work of Orlik and Solomon \cite{OrlikSolomon} who showed that the cohomology rings of $M(\mathcal{A})$ and of $\mathbb{P}(M(\mathcal{A}))$ can be described combinatorially through the lattice of intersections $L(\mathcal{A})$, whose elements are the subspaces given by intersections of hyperplanes in $\mathcal{A}$, ordered by inclusion. 

De Concini and Procesi \cite{DeConciniProcesi} introduced compactifications associated with hyperplane arrangements, called \textit{wonderful models}, inspired by the Deligne--Mumford compactification of the configuration spaces of points of $\C$.
These models depend on the choice of a \textit{building sets}, i.e. a subset $\G \subseteq L(\mathcal{A})$ with nice combinatorial properties. De Concini and Procesi also described the cohomology rings of these compactifications, in terms of the combinatorics of $L(\mathcal{A})$ and $\G$.
Yuzvinsky \cite{Yuzvinsky97} later found an explicit additive basis for these cohomology rings using Gr\"obner basis.

These construnctions were later generalised to a purely combinatorial setting, first to geometric lattices \cite{FeichtnerYuzvinsky} and recently to polymatroids \cite{PagariaPezzoli}. 
In this context, an interesting invariant is the incidence algebra $I(\LL)$; it is the free module with basis indexed by the intervals of the lattice $\LL$.
Given a building set $\G$, we define functions $\overline{\chi}^\G_\LL,H^\G_\LL \in I(\LL)$ which keep track of the reduced characteristic functions and the Chow polynomial respectively of the intervals of the lattice $\LL$ for the building set $\G$. 
In the case of an hyperplane arrangement, i.e. when $\LL=L(\mathcal{A})$, the invariant $\overline{\chi}^\G_\LL$ basically keeps track of the cohomology of the complement of sub-arrangements of contracted arrangements of $\mathcal{A}$ and $H^\G_\LL$ of the cohomology of associated wonderful models. 

Recently, in \cite{EFMPV25}*{Theorem 1.1}, the authors showed that $\overline{\chi}^\G_\LL$ and $H^\G_\LL$ are inverses one to each other in the incidence algebra and, from this general statement 
they expressed the Chow polynomial of $\LL$ and the building set $\G$ in terms of the reduced characteristic functions of its sublattices. From this, they obtained several combinatorial results about the Poincaré polynomial of wonderful models of hyperplane arrangements.

As mentioned at the beginning, in this paper, we are interested in the study of these objects under the action of a finite group $W$, i.e.\ a group with a linear action on $V$ which preserves $\mathcal{A}$ or more generally an action of $W$ on a lattice $\LL$.
In this setting, there is a natural notion of the equivariant incidence algebras $I^W(\LL)$  introduced in \cite{GPY17} and in \cite{EquivariantIncidenceAlgebras}, see Section \ref{chapter-incidence-algebra}  for definitions and notations.
We define functions $\overline{\chi}^{\G}_{\LL}$ and $H^{\G}_{\LL}$ in the equivariant incidence algebra $I^W(\mathcal{L})$, which generalize to the equivariant setting the definitions of \cite{EFMPV25}. Our main theorem for equivariant incidence algebras is the following result, generalizing \cite{EFMPV25}*{Theorem 3.6}.

 \begin{theorem}
We have an equality
\begin{equation}
\label{intro-equation-incidence}
    \overline{\chi}^\G_\LL=(H^\G_\LL)^{-1}.
\end{equation}
 \end{theorem}

Equation (\ref{intro-equation-incidence}) allows us to relate combinatorially the equivariant Chow polynomial and the equivariant reduced characteristic polynomial of the lattice $\LL$, in a way similar to \cite{EFMPV25}. 

In the second part of the paper, we apply this general machinery to the lattices of the hyperplane arrangements of the classical root systems and the action of the corresponding Weyl groups. 

In type $A$, recall that the projectivization of the complement of the hyperplane arrangement of type $A_{n-1}$  and its minimal compactification are $\mathcal{M}_{0,n+1}$ and $\overline{\mathcal{M}}_{0,n+1}$ respectively, i.e.\ the moduli spaces of smooth, respectively stable curves of genus $0$ with $n+1$ marked points.
These moduli spaces are  central and well-studied objects in algebraic geometry, topology and combinatorics.
In 1992, Keel \cite{keel} studied the cohomology of these spaces and obtained a recurrence relation for their Betti numbers.
In 1995, Manin and Getzler showed that is easier and more effective to study the exponential generating series of their Poincaré polynomials than studying each polynomial separately.
Manin \cite{manin} proved a functional equation satisfied by this generating series. 
The symmetric group $S_{n+1}$ acts on $\overline{\mathcal{M}}_{0,n+1}$ by permuting the marked points.
In the same year, Getzler \cite{GetzlerOperads} -- using operadic methods -- established a compositional identity for the corresponding $S_{n+1}$-equivariant generating series $H^A$ encoding the $S_n$ action on the cohomology of $\overline{\mathcal{M}}_{0,n+1}$:
\[H^A(-\overline{\chi}_A)=p_1[X]\]
where $\overline{\chi}_A$ is the $S_{n+1}$-equivariant exponential generating series of the reduced characteristic polynomial of type $A_n$, i.e. the generating series encoding the $S_n$-action on the cohomology of $\mathcal{M}_{0,n+1}$. Using Frobenius characteristic of representations of the groups $S_n$ in the ring of symmetric functions, the composition above is the \textit{plethystic composition} of formal power series with coefficients given by symmetric functions, for more details see Section \ref{compositional-formulas-paragraph}.
Recently, Aluffi, Marcolli, and Nascimento \cite{AluffiMarcolliNascimento25} obtained a closed formula for Poincaré polynomials of $\overline{\mathcal{M}}_{0,n+1}$.
This formula can also be derived by applying the Lagrange inversion to the Getzler's compositional identity \cite{EFMPV25} or by other combinatorial methods, see for instance \cite{cotti}.

In Section  \ref{sec:typeAB}, we show how to use our incidence algebras method to recover the Getzler's result, see Theorem \ref{thm:Getzler_main},   and the Manin's identity, see eq.\ \eqref{manin-equivariant}. It is worth noticing that to show his result, Getzler actually gave a combinatorial description of the whole $S_{n+1}$-action on the cohomology of $\overline{\mathcal{M}}_{0,n+1}$, while we recover just the identities for the restricted $S_n$-action.   Before this result, in Section \ref{compositional-formulas-paragraph}, we also give an explicit and combinatorial formula for the coefficients of the plethystic composition of two formal power series with coefficients in the ring of symmetric functions, see Theorem \ref{compositionalformulaA}. Although this result was likely already known to experts, we were unable to find a specific reference  in the literature. We therefore show it here for the sake of clarity and completeness.

In type $B$, the projectivization of the complement of the arrangement of type $B_n$ and its minimal wonderful model still have an interpretation as moduli spaces in terms of curves of genus $0$ with marked points and involutions, see for instance \cite{HendersonRains08RealDPModelsCoxeter}*{Section 5}. There is significantly less information available in the literature on these varieties. Yuzvinsky \cite{Yuzvinsky97} computed the generating series of the Poincaré polynomials of type $B$ wonderful models associated with the minimal building set. 
In the equivariant setting, i.e.\ keeping track of the hyperoctaedral group $\Z_2 \wr S_n$ action on these spaces, Henderson, using the language of species and an explicit stratification of these moduli spaces, had shown an identity of the following form, see  \cite{Henderson}*{Theorem 2.4}  \begin{equation}
     H^B=\dfrac{-1}{\overline{\chi}^B(H^A[X+Y])}
     \label{intro-eq-typeB}
 \end{equation}
 Recently, Ferroni, Pagaria and Vecchi \cite{Ferroni-Pagaria-Vecchi} found several combinatorial identities that relate the exponential generating series for the Poincaré polynomial of the wonderful model of type $B$ to the reduced characteristic generating series of type $B$ and to the series for the wonderful model of type $A$.

We give an alternative proof of eq.(\ref{intro-eq-typeB}) using incidence algebras, see Theorem \ref{thm:inverted-equation-B}. Moreover, using our equivariant incidence algebra techniques, we show a compositional identity of the following form, see Theorem \ref{identity-typeB}
\begin{equation*}
    \overline{\chi}^B H^B(-\overline{\chi}^A[X+Y])=-1.
\end{equation*}

In type $D$, the Weyl group of type $B$ still acts on the complement of the arrangement of type $D_n$ (which is a subarrangement of that of type $B_n$) and on its minimal wonderful model. For its corresponding equivariant generating series $H^D$, we show an identity of the form, see Theorem~\ref{main-result-typeD} \begin{equation*}
    H^D=\dfrac{-1+xs_{(2,\emptyset)}}{\overline{\chi}^B[H^A[X+Y]]}
\end{equation*}
where $s_{(2,\emptyset)}$ is the characteristic function of the trivial representation of $B_2$. In particular, combined with eq.(\ref{intro-eq-typeB}), we obtain an identity of the following form, for any $n$ \begin{equation}
    H^D_n=H^B_n-x\Ind_{B_2 \times B_{n-2}}^{B_n}(\mathds{1} \boxtimes H^B_{n-2})
\end{equation}
Similar identity in the non-equivariant setting is a consequence of previous results of \cite{Gaiffi-Serventi}.

To show these results, we provide a definition of plethysm in type $B$, that corresponds, using the Frobenius characteristic, to plethym of bisymmetric functions, see Section \ref{sectionplethysm}.
This definition was already partially presented in \cite{Henderson}.
In that article, however, the definitions given are presented via species, and it appears that the representation theoretic interpretation of type $B$ plethystic operations—namely, Theorem \ref{lem:ind_B_j_S_k} and Theorem \ref{lem:ind_S_j_B_k}—are not explicitly presented.
Furthermore,  we provide an explicit combinatorial description of the plethystic composition of two formal power series with coefficients in the ring of bisymmetric functions, see Theorem \ref{compositional-formulatypeB}. 

\section{Equivariant Incidence algebra}
\label{chapter-incidence-algebra}
The equivariant incidence algebras were introduced by Proudfoot \cite{EquivariantIncidenceAlgebras} to study Kazhdan--Lusztig--Stanley polynomials.
We assume, in what follows, that all poset are finite with top $\hat{1}$ and bottom $\hat{0}$ elements.

\begin{definition}
    Let $\LL$ be a finite poset with a compatible action of a finite group $W$.
    Let $\mathcal{C}^W (\LL)$ be the category whose objects are graded $\C$-vector spaces
    \[ V = \bigoplus_{F \leq G} V_{FG}\]
    with an action of the group $W$ such that $\sigma(V_{FG}) \subseteq V_{\sigma(F) \sigma(G)}$.
    The morphisms are $W$-equivariant linear maps that preserve the grading.
    The \emph{equivariant incidence algebra} is the Grothendieck ring $I^W(\LL)$ of the category $\mathcal{C}^W (\LL)$.
\end{definition}

Let $W_F$ be the stabilizer of $F \in \LL$; we denote the intersection of stabilizers by $W_{FG} = W_F \cap W_G$ and similarly for the other indices.
Notice that for every $f \in I^W(\LL)$ the summand $f_{FG}$ belongs to the Grothendieck ring $\Vrep(W_{FG})$ of $\C$-representations of $W_{FG}$.
In \cite{EquivariantIncidenceAlgebras}*{Proposition 2.3}, the category $\mathcal{C}^W (\LL)$ was endowed of a monoidal structure that induces a convolution product on the equivariant incidence algebra given by
\[ (f*g)_{FG} = \sum_{F \leq H \leq G } \frac{\lvert W_{FGH} \rvert}{\lvert W_{FG} \rvert} \Ind_{W_{FGH}}^{W_{FG}}\left( \left( \Res_{W_{FH}}^{W_{FGH}} f_{FH} \right) \otimes \left( \Res_{W_{GH}}^{W_{FGH}} g_{GH} \right) \right).\]

In what follows, we assume that the lattice $\LL$ is geometric with rank function $\rk \colon \LL \to \N$ and that the $W$-action preserves the rank. The \emph{Orlik--Solomon} algebra of a geometric lattices $\LL$ was introduced in \cite{OrlikSolomon}.
It is defined as the exterior algebra over the atoms of $\LL$ with relations generated by 
\[\partial e_C = \sum_{i=1}^{k} (-1)^{i-1} e_{c_1} e_{c_2} \dots e_{c_{i-1}} e_{c_{i+1}} \dots e_{c_k}\]
where $C=(c_1,c_2, \dots, c_k)$ is a circuit, i.e.
\[ \OS_\LL = \faktor{\Lambda[e_c \mid c \textnormal{ atom}]}{(\partial e_C \mid C \textnormal{ circuit})} \]
We denote by $\OS_{\LL,p}$ the degree $p$ part of the Orlik--Solomon algebra.

For any $F \leq G$ in $\LL$ the interval $[F,G]$ is a geometric lattice with the action of the group $W_{FG}$.
Gedeon, Proudfoot and Young \cite{GPY17} defined the equivariant characteristic polynomial of a geometric lattice $\LL$ by 
\[ \chi_\LL(x) = \sum_{p=0}^{\rk(\LL)} (-1)^p x^{\rk(\LL)-p} \OS_{\LL,p} \in \Vrep(W)[x] \]
and in \cite{EquivariantIncidenceAlgebras}*{Proposition 4.1} the corresponding element in $I^W(\LL)\otimes \Z[x]$ was defined as
\[ (\chi_\LL(x))_{FG} = \begin{cases}
    \mathds{1}_W & \textnormal{if } F=G, \\
    \chi_{[F,G]}(x) & \textnormal{otherwise}.
\end{cases} \]

Define the equivariant characteristic polynomial of a polymatroid by
\[ \chi_\LL(x)= \sum_{F \in \LL} \frac{\lvert W_F\rvert}{\lvert W \rvert} \Ind_{W_F}^W \left(\sum_{p} (-1)^{p-1} \tilde{H}_p(\Delta(\hat{0},F)) \right) x^{\rk(\LL)-\rk(F)}. \]
The alternating sum of cohomology groups play the equivariant role of the M\"obius function.
The following lemma can be deduced from \cite{AssafSpeyer}*{Theorem 5.1}
\begin{lemma}
\label{lemma:reduced_chi_is_well_defined}
    For every finite bounded poset $\LL$ and a $W$-action we have
    \[ \sum_{F \in \LL} \frac{\lvert W_F \rvert}{\lvert W \rvert} \Ind_{W_F}^W \chi_{[F,\hat{1}]}(x) = x^{\rk \LL} \mathds{1}_{W_F}\]
    In particular, $\chi_\LL(1)=0$ if $\rk(\LL)>0$.
\end{lemma}

By \cite{GPY17}*{Lemma 2.1} $\chi_\LL(1)=0$ and hence the element 
\[ \overline{\chi}_\LL(x) = \frac{\chi_\LL(x)}{x-1} \in \Vrep (W) [x] \]
is well-defined because $x-1$ is not a zero divisor.
Recall the following definition of a building set $\G$ for the lattice $\LL$.

\begin{definition}
A subset $\G \subseteq \LL \setminus \{\hat{0}\}$ is called building set if, for every $X \in \LL \setminus \{\hat{0}\}$, the morphism given by taking the joint $$\phi_X:\prod_{\substack{\max \text{ } G \in \G \\ G \leq X}}[\hat{0},G] \to [\hat{0},X]$$ is an isomorphism.
\end{definition}

\begin{remark}
For any  lattice $\LL$, we have the \textit{maximal} building set $\G_{\max}=\LL \setminus \{\hat{0}\}$. We also have the \textit{minimal} building set $\mathcal{G}_{\min}$, which consists of the irreducible elements of $\LL$, i.e. of the elements $X \in \LL \setminus \{\hat{0}\}$ such that the interval $[\hat{0},X]$ does not split as a direct product of smaller intervals.    
\end{remark}

Pick a building set $\G$ on the geometric lattice $\LL$ preserved by the $W$-action, i.e. $\sigma(\G) = \G$ for all $\sigma \in W$. Remark that, for any geometric lattice $\LL$ with a compatible $W$-action, the minimal building set $\G_{\min}$ is always preserved by the $W$-action.
For any $F \leq G$ the induced building set $\G^G_F$ on $[F,G]$ is the $W_{FG}$-invariant set
\[ \mathcal{G}_F^G := \{ F \vee H \mid H \in \G, H\leq G, H \not \leq F \}.\]
For any $W$-set $T$ we denote by $\sgn_T$ the sign representation of $W$ for the action on $T$.

The super-reduced characteristic polynomials and functions were introduced in \cite{EFMPV25}*{Definition 3.1 and 3.2}, we generalize them to the equivariant setting.
\begin{definition}
    The \emph{equivariant super-reduced characteristic polynomial} of $\LL$ is 
    \[\overline{\chi}^\G_\LL(x) = (-1)^{\lvert \max \G \rvert -1} \left[ \bigotimes_{G \in \max \G} \overline{\chi}_{[\hat{0},G]}(x) \right] \otimes \sgn_{\max \G} \in \Vrep (W) [x]\]
    with the natural $W$-action on the tensor product and the \emph{equivariant super-reduced characteristic function} is 
    \[(\overline{\chi}^\G_\LL)_{FG} = \begin{cases}
        -\mathds{1}_{W_F} & \textnormal{if } F=G, \\
        \overline{\chi}^{\G_F^G}_{[F,G]}(x) & \textnormal{otherwise}.
    \end{cases} \]
\end{definition}

The Chow ring of a geometric lattice with building set was introduced by De Concini and Procesi \cite{DeConciniProcesi} in the realizable setting, and then extended to atomic lattices by Feichtner and Yuzvinsky \cite{FeichtnerYuzvinsky}.
A generalization to polymatroids appeared in \cite{PagariaPezzoli}.

Recall that the Chow ring (over $\Q$) of a built polymatroid $(\LL,\G)$ is the graded ring
\[\CH_\LL^\G =\dfrac{\Q[x_G]_{G \in \G}}{I(\LL,\G)} \]
where $I(\LL,\G)$ is generated by the elements of the form $$x_{G_1}\cdots x_{G_k}\left(\sum_{H \geq G}x_H\right)^b$$ for each subset $S=\{G_1,\dots,G_m\}$ with $b=\rk(G)-\rk(\sup(S,\G))$, $ \sup(S,\G)=\bigvee_{\substack{G' \in S \\ G' <G}}G'$. Notice that, since the action of $W$ preserves the rank, the action of $W$ naturally extends to an action on $\CH_\LL^{\G}$ (which respects the grading). See also \cite{ANR25} for an equivariant description of the Chow ring of matroids.

\begin{definition}
The \emph{equivariant Chow polynomial} of $(\LL,\G)$ is 
\[H_\LL^\G (x) =\sum_i (H_\LL^\G)^i x^i \in \Vrep(W)[x] \] 
The \emph{equivariant Chow function} of $(\LL,\G)$ is defined by 
\[(H_\LL^\G )_{FG} = H_{[F,G]}^{\G_F^G} (x)\]
\end{definition}

A major result of \cites{Yuzvinsky97,FeichtnerYuzvinsky,PagariaPezzoli} is that the Chow ring admits an explicit Gr\"obner basis, i.e. an additive basis given by a list of monomials.
In particular, we have that a basis of $\CH_{\LL}^\G$ is given by the monomials $$x_{F_1}^{e_1} \cdots x_{F_m}^{e_m}, $$ where $S=\{F_1,\dots,F_m\}$ is $\G$-nested and $$0 \leq e_i \leq \rk(F_i)-\rk(\sup(S,F_i)) .$$

Recall that a subset $S=\{F_1,\dots,F_m\}$ is $\G$-nested if for each incomparable subset $S' \subseteq $ of cardinality at least $2$, the join $\bigvee_{F \in S'}F$ does not belong to $\G$.

Since the action of $W$ is compatible with the lattice structure, every element $w \in W$ sends $\G$-nested sets into $\G$-nested sets. In particular, the action of $W$ on $\CH_{\LL}^{\G}$ is the permutation representations of the action on the elements of the Gr\"obner basis just introduced. We deduce the following Lemma.

\begin{lemma}
The \emph{equivariant Chow polynomial} of $(\LL,\G)$ is 
\[H_\LL^\G (x) = \sum_{S \, \G\textnormal{-nested}} \sum_{g \colon S \to \N} \frac{\lvert W_{Sg} \rvert}{\lvert W \rvert} \Ind_{W_{Sg}}^W x^{\lvert g \rvert} \mathds{1}_{W_{Sg}}, \]
where the sum is taken over all functions $g$ such that $0< g(F) < \rk(F) -\rk (\bigvee (S)_{<F}) $ for all $F \in S$ and the integer $\lvert g \rvert$ is $\sum_{F\in S} g(F)$.
\end{lemma}

\begin{proof}
The result comes from the standard general fact, that for any finite set $X$ with a $W$-action, the character of the corresponding representation $\Q[X]$ is 
\[\Q[X]=\sum_{x \in X}\dfrac{|W_x|}{|W|} \Ind_{W_x}^W(\mathds{1}) . \qedhere \]
\end{proof}

We also introduce the auxiliary function 
\[ \alpha_{\LL}^\G(x)= (-1)^{\lvert \max \G \rvert -1} \sgn_{\max \G} \otimes\sum_{g \colon \max \G \to \N} \frac{\lvert W_{g} \rvert}{\lvert W \rvert} \Ind_{W_{g}}^W x^{\lvert g \rvert} \mathds{1}_{W_{g}}  \in \Vrep(W)[x]\]
where the sum is taken over all functions $g$ such that $0< g(G) < \rk(G) $ for all $G \in \max \G$.
The corresponding element in $I^W(\LL)$ is 
\[ (\alpha_\LL^\G)_{FG} = \alpha_{[F,G]}^{\G_F^G}(x),\]
in particular $(\alpha_\LL^\G)_{FF}=- \mathds{1}_{W_F}$.

Let $\zeta_\LL \in I^W(\LL)$ be the zeta function defined by $ (\zeta_\LL)_{FG}= \mathds{1}_{W_{FG}}.$

\begin{lemma}
\label{lemma:sum_of_signs_vanishes}
    Let $S$ be a finite set with a action of $W$. If $S$ is non-empty then
    \[ \sum_{R \subseteq S} (-1)^{\lvert R \rvert} \frac{\lvert W_R \rvert}{|W|} \Ind_{W_R}^W \sgn_R =0.\]
\end{lemma}
\begin{proof}
    It follows from the exactness of the Koszul complex on $\C^S$.
\end{proof}

The following lemma generalize a key step in the proof of \cite{EFMPV25}*{Theorem 3.6}.

\begin{lemma}
\label{lemma:alpha_H_equal_zeta}
    We have
    \[\alpha^\G_\LL * H_\LL^\G= -\zeta_\LL. \]
\end{lemma}
\begin{proof}
    It is enough to check the following equality for every built lattice
    \[ \sum_{F \in \LL} \frac{\lvert W_F\rvert}{\lvert W \rvert} \Ind_{W_F}^{W} \alpha_{[\hat{0},F]}^{\G^F} \otimes H_{[F, \hat{1}]}^{\G_F} = - \mathds{1}_W.\]
    By definitions $ (-1)^{\lvert f(F) \rvert -1} \sgn_{f(F)} \otimes  \alpha_{[\hat{0},F]}^{\G^F} \otimes H_{[F, \hat{1}]}^{\G_F} $ is equal to
    \begin{align*}
        &
        \left( \sum_{h \colon f(F) \to \N} \frac{\lvert W_{Fh} \rvert}{\lvert W_{F} \rvert} \Ind_{ W_{Fh} } ^{W_F} x^{\lvert h \rvert} \mathds{1}_{W_{Fh}} \right) 
        \otimes 
        \left( \sum_{\substack{T \, \G_F \textnormal{-nested} \\ k \colon T \to \N}} \frac{\lvert W_{FTk} \rvert}{\lvert W_{F} \rvert} \Ind_{ W_{FTk} } ^{W_F} x^{\lvert k \rvert} \mathds{1}_{W_{FTk}} \right).
    \end{align*}
    By standard theory, we have $ (\Ind_{ W_{Fh} } ^{W_F}  \mathds{1}_{W_{Fh}} ) \otimes ( \Ind_{ W_{FTk} } ^{W_F} \mathds{1}_{W_{FTk}} ) = \Ind_{ W_{Fh} } ^{W_F} \Res_{W_{Fh}}^{W_F} \Ind_{ W_{FTk} } ^{W_F} \mathds{1}_{W_{FTk}}$ and 
    $\Res_{W_{Fh}}^{W_F} \Ind_{ W_{FTk} } ^{W_F} \mathds{1}_{W_{FTk}} =  \sum_{\sigma \in W_F} \frac{\lvert W_{Fh} \cap \sigma W_{FTk} \sigma^{-1} \rvert}{\lvert W_{Fh} \rvert \rvert W_{FTk} \rvert} \Ind_{W_{Fh} \cap \sigma W_{FTk} \sigma ^{-1}}^{W_{Fh}} \mathds{1}_{W_{Fh} \cap \sigma W_{FTk} \sigma ^{-1}}$.
    Moreover, $\sigma W_{FTk} \sigma ^{-1}= W_{F \sigma T \sigma k}$ for all $\sigma \in W_F$.
    Hence, the representation $\alpha_{[\hat{0},F]}^{\G^F} \otimes H_{[F, \hat{1}]}^{\G_F} $ is equal to
    \begin{align*}
        &\sum_{h \colon f(F) \to \N} (-1)^{\lvert f(F) \rvert -1} \sgn_{f(F)} \otimes \Ind_{W_{Fh}}^{W_F} \sum_{\substack{T \, \G_F \textnormal{-nested} \\ k \colon T \to \N \\ \sigma \in W_F}} \frac{\lvert W_{Fh\sigma T \sigma k} \rvert}{\lvert W_{F} \rvert^2} \Ind_{W_{Fh \sigma T \sigma k} }^{W_{Fh}} \mathds{1}_{W_{Fh \sigma T \sigma k}} =\\
        &\sum_{h \colon f(F) \to \N} (-1)^{\lvert f(F) \rvert -1} \sgn_{f(F)} \otimes \Ind_{W_{Fh}}^{W_F} \sum_{\substack{T \, \G_F \textnormal{-nested} \\ k \colon T \to \N }} \frac{\lvert W_{Fh T  k} \rvert}{\lvert W_{F} \rvert} \Ind_{W_{Fh  T  k} }^{W_{Fh}} \mathds{1}_{W_{Fh  T  k}}.
    \end{align*}
    Summing over all flats, we have
    \begin{align*}
        \sum_{F \in \LL} \frac{\lvert W_F\rvert}{\lvert W \rvert} \Ind_{W_F}^{W} \alpha_{[\hat{0},F]}^{\G^F} \otimes H_{[F, \hat{1}]}^{\G_F} =
        \sum_{\mathclap{\substack{F \in \LL \\ h \colon f(F) \to \N \\ T \, \G_F \textnormal{-nested} \\ k \colon T \to \N }}} (-1)^{\lvert f(F) \rvert -1} \frac{\lvert W_{Fh T  k} \rvert}{\lvert W \rvert}  \sgn_{f(F)} \otimes \Ind_{W_{Fh  T  k} }^{W_{F}}  x^{\lvert h \rvert +\lvert k \rvert} \mathds{1}_{W_{Fh  T  k}}.
    \end{align*}
    There is a $W$-equivariant bijection between quadruple $F,h,T,k$ with $\G$-nested sets $S$ with function $g \colon S \to \N$ and a subset $R \subseteq \min S$. For more details on this see the proof of \cite{EFMPV25}*{Theorem 3.6}. In particular, $S=T \sqcup f(F)$, $g$ is the gluing of $h$ and $k$ and $R=f(F)$.
    Vice-versa $F= \bigvee R$, $T=S \setminus R$, $h=g_{|R}$ and $k=g_{|S\setminus R}$. The bijection also preserve the stabilizers.
    Finally, we obtain
    \begin{align*}
        \sum_{F \in \LL} \frac{\lvert W_F\rvert}{\lvert W \rvert} \Ind_{W_F}^{W} \alpha_{[\hat{0},F]}^{\G^F} \otimes H_{[F, \hat{1}]}^{\G_F} &= \sum_{\substack{S \, \G \textnormal{-nested} \\ g \colon S \to \N }} \frac{\lvert W_{Sg} \rvert}{\lvert W \rvert}  \Ind_{W_{Sg} }^{W} x^{\lvert g \rvert } 
        \sum_{\mathclap{R \subseteq \min S}} (-1)^{\lvert R \rvert -1} \frac{\lvert W_{SgR} \rvert}{\lvert W_{Sg}\rvert } \Ind_{W_{SgR}}^{W_{Sg}}  \sgn_{R}\\
        &= - \mathds{1}_W
    \end{align*}
    where in the second equality we used \Cref{lemma:sum_of_signs_vanishes} and the only term that does not vanish is associated to $\min S = \emptyset$ that implies $S= \emptyset$.
\end{proof}

\begin{lemma}
\label{lemma:zeta_chi_equal_alpha}
    We have
    \[\zeta_\LL * \overline{\chi}_\LL^\G=\alpha^\G_\LL.\]
\end{lemma}
\begin{proof}
    It is enough to prove
    \[ \sum_{F \in \LL} \frac{\lvert W_F \rvert }{\lvert W \rvert} \Ind_{W_F}^W (-1)^{\lvert \max \G_F \rvert -1} \sgn_{\max \G_F} \otimes \bigotimes_{G \in \max \G_F} \overline{\chi}_{[F,G]} = \alpha_{\LL}^\G(x)\]
    for all built lattices $(\LL,\G)$.
    We put together the summands with the same set $S:= f(F) \cap \max \G$
    \[ \sum_{S \subseteq \max \G} \frac{\lvert W_S \rvert }{\lvert W \rvert} \Ind_{W_S}^W  (-1)^{\lvert \max \G \setminus S \rvert -1} \sgn_{\max \G \setminus S} \otimes \sum_{\substack{F \in \LL \\ f(F) \cap \max \G =S}} \frac{\lvert W_F \rvert }{\lvert W_S \rvert} \Ind_{W_F}^{W_S}  \bigotimes_{G \in \max \G \setminus S} \overline{\chi}_{[F,G]} \]
    where we used the equality $\max \G_F= (\max \G) \setminus S$.
    Moreover, every such $F$ is uniquely determined by the elements $F_i:=F \wedge G_i \in [\hat{0},G_i)$ for all $G_i \in \max \G \setminus S$.
    Hence, the left hand side is equal to
    \begin{align*}
        &\sum_{S \subseteq \max \G} \frac{\lvert W_S \rvert }{\lvert W \rvert} \Ind_{W_S}^W  (-1)^{\lvert \max \G \setminus S \rvert -1} \sgn_{\max \G \setminus S} \otimes 
        \bigotimes_{G_i \in \max \G \setminus S} \sum_{F_i \in [\hat{0},G_i)} \overline{\chi}_{[F_i,G_i]} \\
        &= \sum_{S \subseteq \max \G} \frac{\lvert W_S \rvert }{\lvert W \rvert} \Ind_{W_S}^W  (-1)^{\lvert \max \G \setminus S \rvert -1} \sgn_{\max \G \setminus S} \otimes 
        \bigotimes_{G_i \in \max \G \setminus S} \frac{x^{\rk (G_i)}-1}{x-1} \mathds{1}_{W_{G_i}} \\
        &= \sum_{T \subseteq \max \G} \frac{\lvert W_T \rvert }{\lvert W \rvert} \Ind_{W_T}^W  \left( \sum_{S \subseteq T} \frac{\lvert W_{ST} \rvert }{\lvert W_T \rvert} \Ind_{W_{ST}}^{W_T}  (-1)^{\lvert \max \G \setminus S \rvert -1} \sgn_{\max \G \setminus S} \right) \otimes 
        \, \bigotimes_{\mathclap{G_i \in \max \G \setminus T}} \,\, \frac{x^{\rk (G_i)}-x}{x-1} \mathds{1}_{W_{G_i}} \\
        &= (-1)^{\lvert \max \G \rvert -1 } \sgn_{\max \G} \otimes \bigotimes_{G \in \max \G } \frac{x^{\rk (G)}-x}{x-1} \mathds{1}_{W_{G}}
    \end{align*}
    where we used \Cref{lemma:reduced_chi_is_well_defined} and $T$ is the union of $S$ with the degree zero part of the expansion of 
    $\otimes_{G_i \in \max \G \setminus S} (1+\frac{x^{\rk (G_i)}-x}{x-1} ) \mathds{1}_{W_{G_i}} $.
    By \Cref{lemma:sum_of_signs_vanishes}, the only non-zero summand corresponds to $T=\emptyset$ explaining the last equality.
\end{proof}

\begin{theorem}
\label{theorem_inversion_Chow}    We have
    \[H_\LL^\G = - (\overline{\chi}^\G_\LL)^{-1}. \]
\end{theorem}
\begin{proof}
    It follows from \Cref{lemma:alpha_H_equal_zeta,lemma:zeta_chi_equal_alpha} and the fact that $\zeta_\LL$ is invertible.
\end{proof}

Recall that, if $W$

\section{Weyl groups of type \texorpdfstring{$B$}{B}}
\label{sec:WeylGroups}

For any $n$, we denote by $B_n$ the hyperoctahedral group $B_n=\Z_2^n \rtimes S_n= \Z_2 \wr S_n$, i.e. Weyl group of the root system of type $B$.
We recall here certain general facts about the representation theory and the conjugacy classes in the group $B_n$.
For more details, we refer to \cites{Kerber2,Kerber1}.

\subsection{Representations of \texorpdfstring{$B_n$}{Bn}}

We recall the following standard result in finite group theory.
\begin{theorem}[Mackey Subgroup Theorem]
    Let $H,K$ be two subgroups of a finite group $G$ and choose a set of representative $S \subset G$ for the double coset $K \backslash G/H$.
    For any representation $\rho$ of $H$ and any $s \in S$ let $\rho^s \colon sHs^{-1} \to V$ defined by $\rho^s(shs^{-1})= \rho(h)$. Then
    \[ \Res^G_K \Ind_H^G \rho \cong \bigoplus_{s \in S} \Ind_{K \cap sHs^{-1}}^K \Res^{sHs^{-1}}_{K \cap sHs^{-1}} \rho^s. \]
\end{theorem}

Let now $K,H$ be two finite group with an action $\varphi \colon H \to \Aut(K)$. Denote by $G= K \rtimes_\varphi H $ be the semidirect product.
Let $K^\vee = \Hom(K,\C^*)$ be the character group of $K$ and notice that $H$ acts on $K^\vee$.
For each character $\chi \in K^\vee$ let $H_\chi < H$ be its stabilizer.

Each representation $\chi$ of $K$ extends naturally to a representation of the group $K \rtimes H_\chi$ where the action of $H_\chi$ is trivial.
Moreover, the representations $\rho$ of $H_\chi$ lift to representations of $K \rtimes_\varphi H_\chi$ containing $K$ in the kernel.
We denote by $\chi \otimes \rho$ the tensor product of this two extensions. From Mackey's theorem, we deduce the following.

\begin{theorem}
    The irreducible representations of $G$ over an algebraically closed field are exactly the representations
    \[ \Ind_{K \rtimes_\varphi H_\chi}^G \chi \otimes \rho \]
    where $\chi$ runs over a set of representatives of $H$-conjugate irreducible representations of $K$ and $\rho$ run over the irreducible representations of $H_\chi$.
\end{theorem}
\begin{proof}
    Let us compute 
    \[ \langle \Ind_{K \rtimes H_\chi}^G \chi \otimes \rho , \Ind_{K \rtimes H_{\chi'}}^G \chi' \otimes \rho' \rangle = \langle  \chi \otimes \rho , \Res_{K \rtimes H_\chi}^G \Ind_{K \rtimes H_{\chi'}}^G \chi' \otimes \rho' \rangle \]
    Using the Mackey subgroup theorem we have 
    \[\Res_{K \rtimes H_\chi}^G \Ind_{K \rtimes H_{\chi'}}^G \chi' \otimes \rho' \cong \bigoplus _{s \in S} \Ind_{K \rtimes ( H_\chi \cap sH_{\chi'}s^{-1})}^{K \rtimes  H_\chi } 
    \Res^{K \rtimes sH_{\chi'}s^{-1}}_{K \rtimes ( H_\chi \cap sH_{\chi'}s^{-1})} (\chi' \otimes \rho')^s, \]
    where $S$ is a set of representative of $H_{\chi'}\backslash H /H_\chi$ and the subgroup $K$ is fixed by $s$. Therefore,
    \begin{align*}
        \langle \Ind_{K \rtimes H_\chi}^G \chi \otimes \rho , \Ind_{K \rtimes H_{\chi'}}^G \chi' \otimes \rho' \rangle &= \sum_{s\in S} \langle \Res_{K \rtimes ( H_\chi \cap sH_{\chi'}s^{-1})}^{K \rtimes  H_\chi }  \chi \otimes \rho , \Res^{K \rtimes sH_{\chi'}s^{-1}}_{K \rtimes ( H_\chi \cap sH_{\chi'}s^{-1})} (\chi' \otimes \rho')^s \rangle  \\
        &= \sum_{s\in S} \langle \chi, (\chi')^s \rangle \langle \Res_{H_\chi \cap sH_{\chi'}s^{-1}}^{H_\chi } \rho , \Res^{sH_{\chi'}s^{-1}}_{H_\chi \cap sH_{\chi'}s^{-1}} (\rho')^s \rangle
    \end{align*}
    that is zero if $\chi$ and $\chi'$ are not $H$-conjugate.
    Otherwise, $\chi = (\chi')^s$ and $\Ind_{K \rtimes H_{\chi'}}^G \chi' \otimes \rho' = \Ind_{K \rtimes H_{\chi}}^G \chi \otimes (\rho')^s$ so we can assume $\chi'=\chi$
    We finally have
    \[\langle \Ind_{K \rtimes H_\chi}^G \chi \otimes \rho , \Ind_{K \rtimes H_{\chi}}^G \chi \otimes \rho' \rangle = \langle \chi, \chi \rangle \langle \rho, \rho' \rangle.\]
    In particular the representations $ \Ind_{K \rtimes_\varphi H_\chi}^G \chi \otimes \rho $ are irreducible and pairwise non-isomorphic.

    Moreover, let $V$ be any irreducible representation of $G$ then $\Res^G_K V$ contains an irreducible representation $\chi$ and $H_\chi$ act on the $\chi$-isotypical component of $V$.
    This action contains an irreducible representation $\rho$ of $H_\chi$ and its easy to check that $\Res^G_{K\rtimes H_\chi} V$ contains the representation $\chi \otimes \rho$.
    In particular, $V$ contains the subrepresentation $\Ind_{K \rtimes H_\chi}^G \chi \otimes \rho $.
\end{proof}

In the case of $B_n$, i.e. when $K=\Z^n_2$ and $H=S_n$, a set of representation of $S_n$-orbit of irreducible $\Z_2^n$-representations is given by the characters $1^k \boxtimes (-1)^{n-k}$ for $k=0,\dots, n$, where $1$ and $-1$ are the trivial and sign representation of $\Z_2$

For $\pi \vdash n$ define the $B_n$-representations $V_\pi= 1^n \otimes \Specht_\pi$ and $\tilde{V}_\pi= (-1)^n \otimes \Specht_\pi$.

\begin{lemma}
\label{representationsbn}
    The irreducible representations of $B_n$ are parametrized by pairs of partitions $(\pi, \lambda)$ such that $\lvert \pi \rvert + \lvert \lambda \rvert = n$:
    \[ V_{\pi,\lambda}= \Ind_{B_{k}\times B_{n-k}}^{B_n} V_\pi \boxtimes \tilde{V}_\lambda.\]
\end{lemma}
\begin{proof}
    This is a particular case of the previous theorem. For any $ 1 \leq k \leq n$, the stabilizer in $S_n$ of the $\Z_2^n$-representation $1^k \boxtimes (-1)^{n-k}$ is indeed the subgroup $S_k \times S_{n-k}$ and its irreducible representations are exactly tensor product of two Specht modules.
\end{proof}

Let $\mathcal{P}^{(2)}$ be the set of pairs of partitions. Its element are often called \textit{bipartitions}. Given $\lambda=(\lambda^{(0)},\lambda^{(1)}) \in \mathcal{P}^{(2)}$, we put $|\lambda|=|\lambda^{(0)}|+|\lambda^{(1)}|$ its \textit{size} and we let $\mathcal{P}^{(2)}_n$ be the bipartitions of size $n$. Notice that Lemma \ref{representationsbn} shows that representations of $B_n$ are naturally indexed by elements of $\mathcal{P}^{(2)}_n$.

\subsection{Conjugacy classes of \texorpdfstring{$B_n$}{Bn}}

Let $(\epsilon,\sigma) \in B_n$ with $\epsilon_1,\dots,\epsilon_n \in \Z_2$ and $\sigma \in S_n$. Write $\sigma$ as a product of cycles $\sigma_1 \cdots \sigma_l$ and suppose that for each $j=1,\dots,l$, the cycle $\sigma_j$ is $(i_{j,1} \cdots i_{j,s_j})$. We say that $\sigma_j$ is \textit{positive} if $\epsilon_{i_{j,1}} \cdots \epsilon_{i_{j,s_j}}=1$ and \textit{negative} if $\epsilon_{i_{j,1}} \cdots \epsilon_{i_{j,s_j}}=-1$. We denote by $\sigma^+$ the product of the positive cycles of $\sigma$ and $\sigma^-$ the product of the negative cycles of $\sigma.$ Notice that $\sigma^+\sigma^-=\sigma$ and that the conjugacy class of $\sigma^+,\sigma^-$ are invariant under conjugation in $B_n$.

To the conjugacy class of $\sigma$ we can thus associate a pair of partition $(\lambda^{(0)},\lambda^{(1)})$ such that $|\lambda^{(0)}|+|\lambda^{(1)}|=n$, where $\lambda^{(0)}$ is associated to the conjugacy class of $\sigma^+$ and $\lambda^{(1)}$ is associated to the conjugacy class of $\sigma^-$. The correspondence just introduced above gives an explicit bijection between the conjugacy classes of $B_n$ and $\mathcal{P}^{(2)}_n$.  The cardinality of the centralizer $z_{(\lambda^{(0)},\lambda^{(1)})}$ of the conjugacy class associated to $(\lambda^{(0)},\lambda^{(1)})$ is $$z_{(\lambda^{(0)},\lambda^{(1)})}=2^{l(\lambda^{(0)})}2^{l(\lambda^{(1)})}z_{\lambda^{(0)}}z_{\lambda^{(1)}} .$$

\section{Frobenius characteristic map for \texorpdfstring{$B_n$}{Bn} and bisymmetric functions}
\label{section-frobeniuscharacteristic}
Let $\Lambda(X)$ be the ring of symmetric functions in one infinite set of variables over $\Q$. The Frobenius characteristic map for $S_n$ is 
\[\ch (\rho)= \frac{1}{n!} \sum_{\sigma \in S_n} \tr(\rho(\sigma)) p_\sigma (X) \in \Lambda(X)\]
for any $S_n$-representation $\rho$. It can be proven that $\ch(\Specht_\pi)= s_\pi(X)$. Recall moreover that, for any representations $\chi,\rho$ of $S_k,S_{n-k}$ respectively, we have:
\begin{equation}
    \ch (\Ind_{S_k \times S_{n-k}}^{S_n} \rho \boxtimes \chi )= \ch(\rho) \ch(\chi)
\end{equation}

 The definition of the Frobenius characteristic map can be extended by linearity to any virtual $S_n$-representation and similarly, to any class function $f\colon S_n \to \Q$ as \[\ch (f)= \frac{1}{n!} \sum_{\sigma \in S_n} f(\sigma) p_\sigma (X) \in \Lambda(X)\] For a representation $\rho$, if $\chi_{\rho}$ is its character, we have $\ch(\rho)=\ch(\chi_{\rho})$ and in what follows we will use both notations indistinctly. For more details, see for instance the book of Macdonald \cite{Macdonald}. In what follows, we recall how to construct a similar characteristic map for representations of $B_n$. For more details about this construction, see for instance \cite{Macdonald}*{Appendix B}

Let $\Lambda^{(2)} \coloneqq \Lambda(X) \otimes \Lambda(Y)$ be the ring of symmetric functions in two set of variables $X,Y$ over $\Q$, also called bisymmetric functions.  Basis of $\Lambda^{(2)}$ are naturally indexed by elements of $\mathcal{P}^{(2)}$. In particular, for every $(\lambda,\mu) \in \mathcal{P}^{(2)}$, we put $$s_{(\lambda,\mu)}=s_{\lambda}(X) \otimes s_{\mu}(Y) .$$ The definition of the analogue of the complete symmetric functions for bipartitions is slightly more involved. We denote by $X+Y$ the alphabet in variables $x_i, y_i$. We have thus $p_d[X+Y]=p_d[X]+p_d[Y]$. We write similarly, $p_d[X-Y]=p_d[X]-p_d[Y]$. Similarly, for any $\lambda \in \mathcal{P}$, we write $p_\lambda[X+Y]=\prod_i p_{\lambda_i}[X+Y]$ and $p_\lambda[X-Y]=\prod_i p_{\lambda_i}[X-Y]$. For more about this notation, see next section \ref{sectionplethysm} about \textit{plethysm}.

For any bipartition $(\lambda,\mu) \in \mathcal{P}^{(2)}$, we put then $$p_{(\lambda,\mu)}=p_\lambda[X+Y]p_\mu[X-Y] .$$
Notice, for instance, that, for any $n \in \N$, we have $p_{((n),\emptyset)}=p_n[X+Y]$ and $p_{(\emptyset,(n))}=p_n[X-Y]$. We use \(p_{(\lambda,\mu)}\) when referring to the basis indexed by bipartitions, while \(p_\lambda[X+Y]p_\mu[X-Y]\) is used when the plethystic structure is relevant.

Let now $\chi$ be a representation of $B_n$. 

\begin{definition}
    The \textit{Frobenius characteristic map} in type $B_n$ is 
    \[ \ch(\chi) = \frac{1}{2^n n!}\sum_{\sigma \in B_n} \tr(\chi(\sigma)) p_{\sigma^+}[X+Y]p_{\sigma^-}[X-Y]\]
    where $\sigma^+$ is the partition of associated with positive cycle of $\sigma$ and $\sigma^-$ associated to negative cycles.
\end{definition}

We have a similar extension to any class function $f\colon B_n \to \Q$. The Frobenius characteristic respects the following properties.

\begin{lemma}
    We have:
    \begin{enumerate}
        \item $\ch (\Ind_{B_k \times B_{n-k}}^{B_n} \rho \boxtimes \chi )= \ch(\rho) \ch(\chi)$,
        \item $\ch(V_{\pi,\lambda})=s_\pi(X)s_\lambda(Y)$,
        \item $\ch (\Ind_{S_n}^{B_n} \chi) = \ch(\chi)[X+Y]$.
    \end{enumerate}
\end{lemma}
\begin{proof}
    \begin{enumerate}
        \item We use the formula for the induced character
        \[ \tr(\Ind_{B_k \times B_{n-k}}^{B_n} \rho \boxtimes \chi)(\sigma) = \frac{1}{2^n k! n-k!} \sum_{\substack{\tau \in B_n \\ \tau^{-1} \sigma \tau \in B_k \times B_{n-k}} } \tr((\rho \boxtimes \chi)(\tau^{-1} \sigma \tau)) \]
        and so
        \[\ch (\Ind_{B_k \times B_{n-k}}^{B_n} \rho \boxtimes \chi )= \frac{1}{2^{2n} n! k! n-k!} \sum_{\substack{\sigma, \tau \in B_n \\ \tau^{-1} \sigma \tau \in B_k \times B_{n-k}}} \tr((\rho \boxtimes \chi)(\tau^{-1} \sigma \tau)) p_{\sigma^+}[X+Y]p_{\sigma^-}[X-Y]. \]
        Notice that $\sigma$ and $\tau^{-1} \sigma \tau$ have the same signed cycle type because are conjugate and the formula simplifies to
        \[\ch (\Ind_{B_k \times B_{n-k}}^{B_n} \rho \boxtimes \chi )= \frac{1}{2^{n} k! n-k!} \sum_{ \sigma \in B_k \times B_{n-k}} \tr((\rho \boxtimes \chi)( \sigma )) p_{\sigma^+}[X+Y]p_{\sigma^-}[X-Y]= \ch(\rho) \ch(\chi).\]
        \item By definition of $V_{\pi,\lambda}$ and the first assertion, it is enough to prove that $\ch(V_\pi)=s_\pi(X)$ and $\ch( \tilde{V}_\lambda ) =s_\pi(Y)$.
        We split the signed permutations of $B_n$ accordingly the unsigned underlying permutation of $S_n$:
        \[\ch(V_\pi)= \frac{1}{n!} \sum_{\sigma \in S_n} \tr (\Specht_\pi (\sigma)) \frac{1}{2^n}\prod_{i=1}^{\ell(\sigma)} 2^{\sigma_i-1} (p_{\sigma_i}[X+Y] + p_{\sigma_i}[X-Y]) = \ch(\Specht_\pi)(X) = s_\pi(X). \]
        The same proof holds for $\tilde{V}_\lambda$ with the difference that the character depends on the signed cycle type. Each negative cycle multiply the $S_n$-character by $-1$:
        \[\ch(\tilde{V}_\lambda)= \frac{1}{n!} \sum_{\sigma \in S_n} \tr (\Specht_\lambda (\sigma)) \frac{1}{2^n}\prod_{i=1}^{\ell(\sigma)} 2^{\sigma_i-1} (p_{\sigma_i}[X+Y] - p_{\sigma_i}[X-Y]) = \ch(\Specht_\lambda)(Y) = s_\lambda(Y). \]
        \item We apply the formula for the induced character
        \[ \ch (\Ind_{S_n}^{B_n} \chi) = \frac{1}{2^n (n!)^2} \sum_{\substack{\sigma, \tau \in B_n \\ \tau^{-1} \sigma \tau \in S_n}} \tr(\chi(\tau^{-1} \sigma \tau)) p_{\sigma^+}[X+Y] p_{\sigma^-}[X-Y]\]
        and the permutations $\sigma$ and $\tau^{-1} \sigma \tau \in S_n$ have the same cycle type consisting only in positive cycles.
        We obtain
        \[\ch (\Ind_{S_n}^{B_n} \chi) = \frac{1}{n!} \sum_{ \sigma  \in S_n} \tr(\chi( \sigma )) p_{\sigma}[X+Y] = \ch (\chi)[X+Y]. \qedhere\]
    \end{enumerate}
\end{proof}

\begin{remark}
Recall that, for any $n$, there is a distinct $1$-dimensional representation $\sgn_n^B \colon B_n \to \C^{\times}$, called the \textit{sign} representation, corresponding to the $2$-partition $(\emptyset,(1^n))$. Through the Frobenius characteristic construction introduced above, tensoring with $\sgn_n^B$ in degree $n$ gives an involution $\omega_B \colon \Lambda^{(2)} \to \Lambda^{(2)}$. 

Moreover, there is the $1$-dimensional representation obtained as the composition of the projection $B_n \to S_n$ with the $\sgn$ representation of $S_n$, corresponding thus to the $2$-partition $((1^n),\emptyset)$. As above, tensoring with such a representation gives an involution which we denote by $\omega \colon \Lambda^{(2)} \to \Lambda^{(2)}$.
Notice that, when restricted to the subrings $\Lambda(X),\Lambda(Y)$, the map $\omega$ corresponds to the corresponding canonical involution $\omega$ in the ring of symmetric functions.
\end{remark}

\section{Plethystic operations in type \texorpdfstring{$B$}{B}}
\label{sectionplethysm}
We recall here some generalities about plehtysm for symmetric functions and, more generally, $\lambda$-rings.

\subsection{\texorpdfstring{$\lambda$}{lambda}-rings and plethysm}
We recall the definition and some basic properties of $\lambda$-rings. We refer for instance to \cite{getzler-mhs} for more details.
\begin{definition}

A $\lambda$-ring $R$ is a commutative $\Q$-algebra with homomorphisms $\psi_d \colon R \to R$ for any $d \geq 1$ such that $\psi_{d'}(\psi_d(r))=\psi_{dd'}(r)$ for every $d,d'\in \N_{>0}$ and $r \in R$. The morphisms $\psi_d$ are called  Adams operations. 

A morphism $f\colon R \to R'$ of $\lambda$-rings is a morphism of $\Q$-algebras commuting with the Adams operations.
\end{definition}

For any partition $\mu=(\mu_1,\dots,\mu_h)$, we denote by $\psi_{\mu}\colon R \to R$ the homomorphism defined by $\psi_{\mu}(r)=\psi_{\mu_1}(r)\cdots \psi_{\mu_{h}}(r)$.

\begin{example}
The ring $\Q(t)$ is a $\lambda$-ring , with Adams operations $\psi_d(f(t))=f(t^d)$. Notice that we have $$ \psi_d(t^n)=t^{nd} .$$

The ring $\Lambda(X)$ is a $\lambda$-ring, with Adams operation $\psi_d(f(x_1,\dots,x_n,\dots))=f(x_1^d,\dots,x_n^d,\dots,)$. We remark that we have $$\psi_d(p_n)=p_{nd} .$$
\end{example}
Notice that, for any $\lambda$-ring $R$ and any element $r \in R$, we can define a unique $\lambda$-ring homomorphism $\Lambda(X) \to R$ as $$p_{\lambda} \to \psi_{\lambda}(r) .$$

More precisely, for any $\lambda$-ring $R$, we can define a \textit{plethysm} linear on the first component $$[-]\colon \Lambda(X) \times R \to R $$ such that we have 
\begin{itemize}
\item For any $r \in R$, the map $[-]\circ r \colon \Lambda(X) \to R$ is a ring homomorphism.
\item $p_d[a]=\psi_d(a)$.
\end{itemize}

In particular, given the $\lambda$-ring structure of $\Lambda(X)$ introduced above, for any $f,g \in \Lambda(X)$, we can define their \textit{plethystic composition} $f[g]$. Notice that this latter operation shares the following properties \begin{flalign*}
    &(af+bh)[g]=af[g]+bh[g] \text{ for any } f,h \in \Lambda(X) \text{ and } a,b \in \Q \\
    &(fh)[g]=f[g]h[g]\\
    &p_n[p_m]=p_{nm} \text{ for any } n,m\\
    &p_n[f+g]=p_n[f]+p_n[g]
\end{flalign*}

\begin{remark}
\label{remark-plethysm-minus}
While $p_n[-g]=-p_n[g]$ for any $g \in R$, for an homogenous $f \in \Lambda$ of degree $n$, we find in general 
\begin{equation}
    f[-g]=(-1)^{n}\omega(f)[g]
    \label{plethysm-minus}
\end{equation}    
\end{remark}

It is possible to show that, for any $f,g \in R$, for any $k \in \N$ and any $\rho \in \mathcal{P}_k$, we have \begin{align}
\label{plethysm-schur-product}
   & s_\rho [fg] = \sum_{\alpha,\beta \vdash k} g_{\alpha\beta\rho} s_\alpha[f] s_\beta[g]\\
   & s_{\rho}[f+g]=\sum_{\substack{\pi,\lambda \\|\pi|+|\lambda|=k }}c^{\rho}_{\pi,\lambda}s_{\pi}[f]s_{\lambda}[g]
  \end{align}
where $g_{\alpha\beta\rho} ,c^{\rho}_{\pi,\lambda}$ are respectively the Kronecker and Littlewood-Richardson coefficients. We sketch for instance a proof of the first identity. The second one is proven similarly.

The left hand side is indeed $$\sum_{\mu \in \mathcal{P}_k} z_{\mu}^{-1}\chi^{\lambda}_{\mu}p_{\mu}[f]p_{\mu}[g] $$ and the right hand side is $$\sum_{\mu_1,\mu_2 \in \mathcal{P}_k}\sum_{\alpha,\beta \vdash k}g_{\alpha,\beta,\rho}\chi^{\alpha}_{\mu_1}\chi^{\beta}_{\mu_2}z_{\mu_1}^{-1}z_{\mu_2}^{-1}p_{\mu_1}[f]p_{\mu_2}[g] .$$ The first identity of eq.\ \eqref{plethysm-schur-product} is now a consequence of the following identity:
$$\sum_{\alpha,\beta \vdash k}g_{\alpha,\beta,\rho}\chi^{\alpha}_{\mu_1}\chi^{\beta}_{\mu_2}z_{\mu_1}^{-1}z_{\mu_2}^{-1}=\delta_{\mu_1,\mu_2}\chi^{\lambda}_{\mu_1}z_{\mu_1}^{-1} $$ which is a consequence of the orthogonal relations for the characters of $S_n$.

Recall that the plethysm of symmetric function has the following representation theoretic interpretation. Given $\rho,\chi$ irreducible representations of $S_k,S_l$ respectively, we have the following identity.
\begin{equation}
\label{plethysm-for-reps}
    \ch( \Ind_{S_k \wr S_l}^{S_{kl}} \rho^{\boxtimes l} \otimes \chi) = \ch(\chi)[\ch(\rho)]
\end{equation}

\subsubsection{Plethysm of formal power series}
\label{paragraph-plethysm-series}
For a $\lambda$-ring $R$, the ring $R[[t]]$ of formal power series with coefficients in $R$ has a natural $\lambda$-ring structure with $\psi_d(\sum_n f_nt^n)=\sum_n \psi_d(f_n)t^{nd}$. For any $f \in tR[[t]]$, we put $$\Exp(f) \coloneqq \exp\left(\sum_{d \geq 1}\dfrac{\psi_d(f)}{d}\right) .$$  Notice that the term on the right hand side is well defined since $f_0=0$. The function $\Exp:tR[[t]] \to 1+tR[[t]]$ is called \textit{plethystic exponential.}
 The plethystic exponential admits an inverse map $\Logg:1+tR[[t]] \to tR[[t]]$, called plethystic logarithm. For more details see for instance \cite{getzler-mhs}. Given $f \in 1+tR[[t]]$ and $r \in R$, we put $$f^{\langle r \rangle}\coloneqq \Exp(\Logg(f)r) .$$

\subsection{Plethysm of bisymmetric  functions}

Notice that the ring $\Lambda(X) \otimes \Lambda(Y)$ has a natural $\lambda$-ring structure, with $p_d[f(x_i,y_j)_{i,j}]=f(x_i^d,y_j^d)_{i,j}$. Notice that, for any $(\lambda,\mu) \in \mathcal{P}^{(2)}$, we have \begin{equation}
    p_d[p_{(\lambda,\mu)}]=p_{(d\lambda,d\mu)}.
\end{equation}

When dealing with plethysm of bisymmetric functions, we will  sometimes use the following convention. We will, in fact, work with the ring $\Lambda(X) \otimes \Lambda(Y) \otimes \Q[\epsilon]$ with the $\lambda$-ring structure defined by putting $p_k[\epsilon]=\epsilon^k$. The variable  $\epsilon$ will always ultimately be always evaluated at $\epsilon=-1$. This apparent complication is introduced because plethysm does not commute with evaluation, see for instance Remark \ref{remark-plethysm-minus}.

\begin{remark}
 Notice that, we can think of $X, Y$ as elements of $\Lambda^{(2)}$, corresponding to $X=p_1[X]$ and $Y=p_1[Y]$ and $X+Y=p_1[X]+p_1[Y]=p_1[X+Y]$. For the $\lambda$-ring structure just introduced, we have $p_n[X+Y]=p_n[p_1[X]+p_1[Y]]=p_n[X]+p_n[Y]$ and similarly for $X-Y$ and $p_n[X-Y]$ which explains the notation introduced at the beginning of \cref{section-frobeniuscharacteristic}.
\end{remark}

In the following paragraph, we are going to describe certain properties of the $\lambda$-ring structure just introduced and then show how to compose plethystically two bisymmetric functions.

We start by remarking the following properties.

    \begin{lemma}
    \label{lem:ind_Z_2xS_n}
    We have
        \[\ch(\Ind_{\Z_2 \times S_n}^{B_n} 1 \boxtimes V_\rho) = \sum_{ \substack{\lvert \pi \rvert +  \lvert \lambda \rvert =n \\ \lvert \lambda \rvert \textnormal{ even}}} c_{\pi, \lambda}^\rho s_\pi(X) s_\lambda(Y) = \frac{1}{2}(s_\rho[X+Y]+s_\rho[X+\epsilon Y])_{|\epsilon=-1} \]
        and 
        \[\ch(\Ind_{\Z_2 \times S_n}^{B_n} (-1) \boxtimes V_\rho) = \sum_{ \substack{\lvert \pi \rvert +  \lvert \lambda \rvert =n \\ \lvert \lambda \rvert \textnormal{ odd}}} c_{\pi, \lambda}^\rho s_\pi(X) s_\lambda(Y) = \frac{1}{2}(s_\rho[X+Y] - s_\rho[X+\epsilon Y])_{|\epsilon=-1} \]
        where $\Z_2 \simeq Z(B_n)$ and $c_{\pi,\lambda}^{\rho}$ are the Littlewood–Richardson coefficients.

        More generally, for reducible representations $V$ we have
        \[\ch(\Ind_{\Z_2 \times S_n}^{B_n} 1 \boxtimes V) = \frac{1}{2}(\ch(V)[X+Y]+\ch(V)[X+\epsilon Y])_{|\epsilon=-1} \]
        and 
        \[\ch(\Ind_{\Z_2 \times S_n}^{B_n} (-1) \boxtimes V) = \frac{1}{2}(\ch(V)[X+Y] - \ch(V)[X+\epsilon Y])_{|\epsilon=-1} \]
    \end{lemma}
    \begin{proof}
    We use the Mackey subgroup theorem to obtain
    \begin{align*}
        \langle \Ind_{\Z_2 \times S_n}^{B_n} 1 \boxtimes V_\rho , V_{\pi,\lambda}\rangle &= \langle  1 \boxtimes V_\rho, \Res_{\Z_2 \times S_n}^{B_n} V_{\pi,\lambda}\rangle \\
        & = \langle  1 \boxtimes V_\rho, \Ind^{\Z_2 \times S_n} _{\Z_2 \times S_k \times S_{n-k}} \Res_{\Z_2 \times S_k \times S_{n-k}} ^{B_k \times B_{n-k}} V_{\pi,\lambda}\rangle \\
        & = \langle \Res^{\Z_2 \times S_n} _{\Z_2 \times S_k \times S_{n-k}}  1 \boxtimes V_\rho, \Res_{\Z_2 \times S_k \times S_{n-k}} ^{B_k \times B_{n-k}} V_{\pi,\lambda}\rangle \\
        &= \langle 1, (-1)^{n-k} \rangle \langle \Res^{S_n} _{S_k \times S_{n-k}} \Specht_\rho, \Specht_\pi \boxtimes \Specht_\lambda \rangle.
    \end{align*}
    That is zero if $\lvert \lambda \rvert$ is odd and otherwise equal to $c_{\pi,\lambda}^{\rho}$. 
    
    The second identity follows by a similar proof or by:
        \[ \ch(\Ind_{\Z_2 \times S_n}^{B_n} 1 \boxtimes V_\rho) + \ch(\Ind_{\Z_2 \times S_n}^{B_n} (-1) \boxtimes V_\rho) = s_\rho [X+Y]
        = \sum_{\pi, \lambda} c_{\pi,\lambda}^{\rho} s_\pi(X)s_\lambda(Y), \] where the last equality is a direct consequence of eq.\ \eqref{plethysm-schur-product}.
        The case of reducible representations follows by linearity.
    \end{proof}

\begin{lemma}
For any $\rho,\pi,\lambda \in \mathcal{P}$, we have:
\label{lem:ind_B_j_S_k}
    \[
        \ch( \Ind_{B_j \wr S_k}^{B_{jk}} V_{\pi,\lambda}^{\boxtimes k} \otimes V_\rho) = s_\rho[s_\pi(X)  s_\lambda(Y)].
    \]
\end{lemma}
\begin{proof}
    Let $l= \lvert \pi \rvert$ and notice that 
    \[ \Ind_{B_j \wr S_k}^{B_{jk}} V_{\pi,\lambda}^{\boxtimes k} \otimes V_\rho = \Ind_{(B_l \times B_{j-l}) \wr S_k}^{B_{jk}} V_{\pi}^{\boxtimes k} \boxtimes \tilde{V}_{\lambda}^{\boxtimes k} \otimes V_\rho \]
    and that 
    \[\Ind_{(B_l \times B_{j-l}) \wr S_k}^{B_{l} \wr S_k \times B_{j-l} \wr S_k} V_{\pi}^{\boxtimes k}\boxtimes \tilde{V}_{\lambda}^{\boxtimes k} \otimes V_\rho = \bigoplus_{\alpha,\beta \vdash k} ((V_{\pi, \emptyset}^{\boxtimes k} \otimes V_\alpha) \boxtimes (\tilde{V}_{\emptyset,\lambda}^{\boxtimes k} \otimes V_\beta))^{\oplus g_{\alpha\beta\rho}}\]
    where $S_k$ is embedded in $B_{l} \wr S_k \times B_{j-l} \wr S_k$ diagonally, essentially because $\Ind_{S_k}^{S_k \times S_k} V_\rho = \oplus_{\alpha,\beta \vdash k} (V_\alpha \boxtimes V_\beta)^{\oplus g_{\alpha\beta\rho}}$, see eq.\ \eqref{plethysm-schur-product}.
    In particular,
    \[ \Ind_{B_j \wr S_k}^{B_{jk}} V_{\pi,\lambda}^{\boxtimes k} \otimes V_\rho = \bigoplus_{\alpha,\beta \vdash k} \left( \Ind^{B_{jk}}_{B_{l} \wr S_k \times B_{j-l} \wr S_k} (V_{\pi, \emptyset}^{\boxtimes k}\boxtimes V_\alpha) \boxtimes (\tilde{V}_{\emptyset,\lambda}^{\boxtimes k} \boxtimes V_\beta)\right)^{\oplus g_{\alpha\beta\rho}}\]
    By taking the Frobenius characteristic we have
    \[\ch( \Ind_{B_j \wr S_k}^{B_{jk}} V_{\pi,\lambda}^{\boxtimes k} \otimes V_\rho) = \sum_{\alpha,\beta \vdash k} g_{\alpha\beta\rho} \ch \left(\Ind^{B_{lk}}_{B_{l} \wr S_k } V_{\pi, \emptyset}^{\boxtimes k}\otimes V_\alpha \right) \ch \left(\Ind^{B_{(j-l)k}}_{B_{j-l} \wr S_k } V_{ \emptyset, \lambda}^{\boxtimes k}\otimes V_\beta \right).  \]

    Now it is enough to prove the claimed identity for $\lambda= \emptyset$ or $\pi=\emptyset$.
    Clearly, $\ch \left(\Ind^{B_{lk}}_{B_{l} \wr S_k } V_{\pi, \emptyset}^{\boxtimes k}\otimes V_\alpha \right) \in \Lambda(X)$ and  $\ch \left(\Ind^{B_{(j-l)k}}_{B_{j-l} \wr S_k } V_{ \emptyset, \lambda}^{\boxtimes k}\otimes V_\beta \right) \in \Lambda(Y)$ and so it is enough to compute their restriction to $S_n$.
    This reduces to the standard plethysm for the symmetric group.
\end{proof}

The Lemma above implies the following Theorem, which is the analogue of eq.(\ref{plethysm-for-reps}) in this context.

\begin{theorem}
\label{thm:Ind_repr_wr0}
    For every representation $V$ of $S_k$ and $W$ of $B_j$ we have
    \begin{align}
        & \ch( \Ind_{B_j \wr S_k}^{B_{jk}} W^{\boxtimes k} \otimes V) = \ch(V)[\ch(W)].
    \end{align}
\end{theorem}
\begin{proof}
Since the wanted equality is linear in the entry $V$, it is enough to show this for any basis of the class functions of $S_k$. In particular, it is enough to show that $ \ch( \Ind_{B_j \wr S_k}^{B_{jk}} \chi_W^{k} \otimes 1_{C_{\lambda}})=\ch(1_{C_{\lambda}})[\ch(W)]$ for every $\lambda$ partition of $k$, where $C_{\lambda}$ is the corresponding conjugacy class of $S_{k}$ and $1_{C_{\lambda}}$ its characteristic function. Recall that $\ch(1_{C_{\lambda}})=p_{\lambda}/z_{\lambda}$. Moreover, if $\lambda=(\lambda_1,\dots,\lambda_r)$, we have $$p_{\lambda}[\ch(W)]=\prod_{i=1}^r p_{\lambda_i}[\ch(W)] .$$

Notice that, for any two representations $W_1,W_2$ of $B_j$ and any $d \in \mathbb{N}_{>0}$, we have that the following equality of class functions \begin{equation}(\chi_{W_1} \oplus \chi_{W_2})^{d} \otimes 1_{C_{(d)}}=\chi_{W_1}^d \otimes 1_{C_{(d)}} \oplus \chi_{W_2}^d \otimes 1_{C_{(d)}} .\label{linearityplethysm}\end{equation}

Eq.(\ref{linearityplethysm}) above is not true in general, replacing $1_{C_{(d)}}$ with another class function. For the function $1_{C_{(d)}}$ this comes from the fact that a cycle of length $d$ acts transitively on the set $\{1,\dots,d\}$.

From \Cref{lem:ind_B_j_S_k} and the linearity of the plethysm $p_d[-]$, we deduce that for any representation $W$ of $B_j$ and any $d$, we have $$\ch( \Ind_{B_j \wr S_d}^{B_{jd}} \chi_W^{k} \otimes z_{(d)}1_{C_{(d)}})=p_d[\ch(W)] .$$

Notice now that we have $$1_{C_{\lambda}}=\dfrac{z_{(\lambda_1)} \cdots z_{(\lambda_r)}}{z_{\lambda}} \Ind_{S_{\lambda_1} \times \cdots \times S_{\lambda_r}}^{S_k}(1_{C_{(\lambda_1)}} \boxtimes \cdots \boxtimes 1_{C_{(\lambda_r)}}) .$$
We have therefore \begin{align*}&\Ind_{B_j \wr S_k}^{B_{jk}} \chi_W^{k} \otimes 1_{C_{\lambda}}=\\ &=\dfrac{z_{(\lambda_1)} \cdots z_{(\lambda_r)}}{z_{\lambda}}\Ind_{B_j \wr S_k}^{B_{jk}} \Ind^{B_j \wr S_k}_{(B_j \wr S_{\lambda_1}) \times \cdots (B_j \wr \times S_{\lambda_r})}(\chi_W^{\lambda_1} \otimes 1_{C_{(\lambda_1)}} \boxtimes \cdots \boxtimes \chi_W^{\lambda_r} \otimes 1_{C_{(\lambda_r)}} )=\\ &=\dfrac{z_{(\lambda_1)} \cdots z_{(\lambda_r)}}{z_{\lambda}} \Ind^{B_{jk}}_{B_{j \lambda_1} \times \cdots \times B_{j \lambda_r}} \Ind^{B_{j \lambda_1} \times \cdots \times B_{j \lambda_r}}_{(B_j \wr S_{\lambda_1}) \times \cdots (B_j \wr \times S_{\lambda_r})}(\chi_W^{\lambda_1} \otimes 1_{C_{(\lambda_1)}} \boxtimes \cdots \boxtimes \chi_W^{\lambda_r} \otimes 1_{C_{(\lambda_r)}} )=\\ &=\prod_{i=1}^r \ch(\Ind^{B_{j\lambda_i}}_{B_j \times S_{\lambda_i}}(\chi_W^{ \lambda_i} \otimes 1_{C_{(\lambda_i)}}))=\dfrac{p_{\lambda}[\ch(W)]}{z_{\lambda}} . \qedhere
\end{align*}
\end{proof}

\subsubsection{Plethystic composition of bisymmetric functions}

For every symmetric function $f \in \Lambda^{(2)}$ we denote $f^{\even}$ and $f^\odd$ the even and odd part with respect to the $Y$-grading.
Define $$(f_1\otimes f_2)[g] \coloneqq f_1[g^\even] f_2[g^\odd] \in \Lambda(X)\otimes \Lambda(Y)$$ for every $f_1 \in \Lambda(X)$ and $f_2 \in \Lambda(Y)$ and extend it by linearity to $f[g]$ for every $f,g \in \Lambda(X) \otimes \Lambda(Y)$. Observe that if $f,g \in \Lambda(X) \subset \Lambda(X) \otimes \Lambda(Y)$ this new operation $f[g]$ agrees with the classical plethysm.

Fix $f \in \Lambda(X) \otimes \Lambda(Y)$, for every $d \in \N$, we compute now  $p_{(d,\emptyset)}[f]$ and $p_{(\emptyset,d)}[f]$. We have that $p_{(d,\emptyset)}=p_d[X]+p_d[Y]$ and $p_{(\emptyset,d)}=p_d[X]-p_d[Y]$. We have thus $$p_{(d,\emptyset)}[f]=p_d[f^{\even}]+p_d[f^{\odd}]=p_d[f]$$ and $$p_{(\emptyset,d)}[f]=p_d[f^{\even}]-p_d[f^{\odd}]=p_d[f^{\even}-f^{\odd}] .$$ In particular, both the operations $p_{(d,\emptyset)}[-]$ and $p_{(\emptyset,d)}[-]$ are linear.

More generally, we deduce that, for any $(\alpha,\beta) \in \mathcal{P}^{(2)}$, we have  \begin{equation}p_{(\alpha,\beta)}[f]=\prod_{i=1}^{l(\alpha)} p_{((\alpha_i),\emptyset)}[f] \prod_{l=1}^{l(\beta)} p_{(\emptyset,(\beta_l))}[f]=p_{\alpha}[f]p_{\beta}[f^{\even}-f^{\odd}] .\end{equation}

In this paragraph, we show certain properties of this operation on bisymmetric functions. We start by showing the following result.

\begin{lemma}
    \label{lem:ind_S_j_B_k}
    We have
    \begin{align*}
        \ch( \Ind_{S_j \wr B_k}^{B_{jk}} V_\rho^{\boxtimes k} \otimes V_{\pi,\lambda}) &= s_\pi \left[ \sum_{ \substack{\lvert \alpha \rvert +  \lvert \beta \rvert =j \\ \lvert \beta \rvert \textnormal{ even}}} c_{\alpha, \beta}^\rho s_\alpha(X) s_\beta(Y) \right] s_\lambda \left[ \sum_{ \substack{\lvert \alpha \rvert +  \lvert \beta \rvert =j \\ \lvert \beta \rvert \textnormal{ odd}}} c_{\alpha, \beta}^\rho s_\alpha(X) s_\beta(Y) \right] \\
        &= (s_\pi \otimes s_\lambda)[s_\rho[X+Y]]
    \end{align*} 
where the action of $B_k$ on $S_j^k$ factors through $S_k$.
\end{lemma}
\begin{proof}
    Let $l= \lvert \pi \rvert$ and notice that 
    \[ \Ind_{S_j \wr B_k}^{B_{jk}} V_\rho^{\boxtimes k} \otimes V_{\pi,\lambda} = \Ind_{S_j \wr B_l \times S_j \wr B_{k-l}}^{B_{jk}} V_\rho^{\boxtimes l} \otimes V_{\pi} \boxtimes V_\rho^{\boxtimes l} \otimes \tilde{V}_{\lambda} \]
    and so the characteristic is
    \[ \ch( \Ind_{S_j \wr B_k}^{B_{jk}} V_\rho^{\boxtimes k} \otimes V_{\pi,\lambda}) = \ch(\Ind_{S_j \wr B_l }^{B_{jl}} V_\rho^{\boxtimes l} \otimes V_{\pi}) \ch(\Ind_{S_j \wr B_{k-l}}^{B_{j(k-l)}}  V_\rho^{\boxtimes l} \otimes \tilde{V}_{\lambda}).\]
    We deal with the first factor, the second one is similar:
    \begin{align*}
    \Ind_{S_j \wr B_l }^{B_{jl}} V_\rho^{\boxtimes l} \otimes V_{\pi} 
    &=  \Ind_{B_j \wr S_l }^{B_{jl}} \Ind_{(\Z_2 \times S_j) \wr S_l}^{B_j \wr S_l} (1 \boxtimes V_\rho)^{\boxtimes l} \otimes V_{\pi} \\
    &= \Ind_{B_j \wr S_l }^{B_{jl}} \left( \Ind_{\Z_2 \times S_j}^{B_j } 1 \boxtimes V_\rho \right)^{\boxtimes l} \otimes V_{\pi}
    \end{align*}
    and using \Cref{thm:Ind_repr_wr0} and \Cref{lem:ind_Z_2xS_n} we obtain the claimed equality. 
\end{proof}

We can generalize Lemma \ref{lem:ind_S_j_B_k} by the following Theorem.

\begin{theorem}
\label{thm:Ind_repr_wr}
For every representation $V$ of $S_j$ and $W$ of $B_k$, we have
\begin{equation}\ch( \Ind_{S_j \wr B_k}^{B_{jk}} V^{\boxtimes k} \otimes W) = \ch(W)[\ch(V)[X+Y]]\end{equation} where the action of $B_k$ factorizes through $S_k$.
\end{theorem}

\begin{proof}

Consider a $2$-partition $(\alpha,\beta) \in \mathcal{P}_n^{(2)}$ and the associated conjugacy class $C_{(\alpha,\beta)}$ of $B_n$. Assume that $\alpha=(\alpha_1,\dots,\alpha_s)$ and $\beta=(\beta_1,\dots,\beta_t)$. Recall that we have $\ch(1_{C_{(\alpha,\beta)}})=\dfrac{p_{(\alpha,\beta)}}{z_{\alpha,\beta}}$.

Since in the wreath product $S_j \wr B_d$ the action of $B_d$ factorizes through that of $S_d$, we see as before that, for any $W_1,W_2$ representations of $S_j$ we have the following equalities of class functions $$(\chi_{W_1} \oplus \chi_{W_2})^{d} \otimes 1_{C_{((d),\emptyset)}}=\chi_{W_1}^d \otimes 1_{C_{((d),\emptyset)}} \oplus \chi_{W_2}^d \otimes 1_{C_{((d),\emptyset)}} $$ and  $$(\chi_{W_1} \oplus \chi_{W_2})^{d} \otimes 1_{C_{(\emptyset,(d))}}=\chi_{W_1}^d \otimes 1_{C_{(\emptyset,(d))}} \oplus \chi_{W_2}^d \otimes 1_{C_{(\emptyset,(d))}} .$$

From \Cref{lem:ind_S_j_B_k}, we see that for any representation $W$ of $S_j$ and any $d$, we have $$\ch( \Ind_{S_j \wr S_d}^{B_{jd}} \chi_W^{k} \otimes z_{((d),\emptyset)}1_{C_{(0,(d))}})=p_{(d,\emptyset)}[\ch(W)[X+Y]] $$ and  $$\ch( \Ind_{S_j \wr B_d}^{B_{jd}} \chi_W^{k} \otimes z_{(\emptyset,(d))}1_{C_{(\emptyset,(d))}})=p_{(\emptyset,(d))}[\ch(W)[X+Y]] .$$

Put $$z'_{(\alpha,\beta)}\coloneqq \dfrac{\prod_{i,j}z_{(\alpha_i,\emptyset)}z_{(\emptyset,\beta_j)}}{z_{(\alpha,\beta)}}.$$ Notice that we have \begin{align*}&1_{C_{(\alpha,\beta)}}=\\ &z'_{(\alpha,\beta)}\Ind_{\prod_{i,j}B_{\alpha_i} \times B_{\beta_j}}^{B_k}(1_{C_{((\alpha_1),\emptyset)}} \times \cdots \times 1_{C_{((\alpha_s),\emptyset)}} \times 1_{C_{(\emptyset,(\beta_1))}} \times \cdots \times 1_{C_{(\emptyset,(\beta_t))}}) .\end{align*}

We have therefore that \begin{align}& \Ind_{S_j \wr B_k}^{B_{jk}} \chi_V^{ k} \otimes 1_{C_{(\alpha,\beta)}}=\nonumber \\ &z'_{(\alpha,\beta)} \Ind^{S_j \wr b_k}_{(S_j \wr B_{\alpha_1}) \times \cdots (S_j \wr \times B_{\alpha_s}) \times (S_j \wr B_{\beta_1}) \times \cdots (S_j \wr B_{\beta_t})}((\chi_V^{\alpha_1} \otimes 1_{C_{((\alpha_1),\emptyset)}}) \boxtimes \cdots \boxtimes (\chi_V^{\beta_t} \otimes 1_{C_{(\emptyset,(\beta_t))}}))\label{eq-proof-composition-plethysm}\end{align}

Put $\overline{B}_{jk}=B_{j\alpha_1} \times \cdots B_{j \alpha_k} \times B_{j\beta_1} \times \cdots \times B_{j \beta_t}$. We have thus that the RHS of eq.(\ref{eq-proof-composition-plethysm}) is equal to 
\begin{align*} 
&=z'_{(\alpha,\beta)} \Ind^{B_{jk}}_{\overline{B}_{jk}}\Ind^{\overline{B}_{jk}}_{\prod_{i=1}^s (S_j \wr B_{\alpha_i}) \prod_{l=1}^t (S_j \wr \beta_l)}(\boxtimes_{i=1}^r (\chi_V^{\alpha_i} \otimes 1_{C_{((\alpha_i),\emptyset)}}) \boxtimes_{l=1}^t \chi_V^{\beta_l} \otimes 1_{C_{(\emptyset,(\beta_l))}})=\end{align*} and taking characters we have thus \begin{align*}&\dfrac{p_{(\alpha,\beta)}}{z_{(\alpha,\beta)}}[\ch(V)[X+Y]]=\prod_{i=1}^s \ch(\Ind^{B_{j \alpha_i}}_{S_j \wr B_{\alpha_i}}(\chi_V^{\alpha_i} \otimes 1_{C_{((\alpha_i),\emptyset)}})) \prod_{l=1}^t\ch(\Ind^{B_{j \beta_l}}_{S_j \wr B_{\beta_l}}(\chi_V^{\beta_l} \otimes 1_{C_{(\emptyset,(\beta_l))}})). \qedhere \end{align*}
\end{proof}

We end the section by showing the following technical Lemma.

\begin{lemma}
\label{plethysmsn-bn}
For any two symmetric functions $f,g \in \Lambda(X)$, we have that \begin{equation}
\label{plethysmsn-bn-eq}
    f[g][X+Y]=f[X+Y][g[X+Y]]
\end{equation}    
\end{lemma}

\begin{proof}
From the linearity of eq.(\ref{plethysmsn-bn-eq}) in $f$, it is enough to show it for $f=p_{\lambda}$ with $\lambda \in \mathcal{P}$. Put $g=\sum_{\mu \in \mathcal{P}}g_{\mu}p_{\mu}$. We have thus \begin{align*}
    &p_{\lambda}\left[\sum_{\mu \in \mathcal{P}} g_{\mu}p_{\mu}\right][X+Y]=\prod_{i=1}^{l(\lambda)}p_{\lambda_i}\left[\sum_{\mu \in \mathcal{P}} g_{\mu}p_{\mu}\right][X+Y]=\\ &\prod_{i=1}^{l(\lambda)}\left(\sum_{\mu \in \mathcal{P}} g_{\mu}p_{\lambda_i\mu}\right)[X+Y]=\\ &\prod_{i=1}^{l(\lambda)}\left(\sum_{\mu \in \mathcal{P}} g_{\mu}p_{(\lambda_i\mu,\emptyset)}\right).
\end{align*}     

We have moreover $p_{\lambda}[X+Y]=p_{(\lambda,\emptyset)}$ and $\displaystyle g[X+Y]=\sum_{\mu \in \mathcal{P}}g_{\mu}p_{(\mu,\emptyset)}$ and thus $$
   p_{\lambda}[X+Y]\left[\sum_{\mu \in \mathcal{P}}g_{\mu}p_{(\mu,\emptyset)}\right]=\prod_{i=1}^{l(\lambda)}p_{((\lambda_i),\emptyset)}\left[\sum_{\mu \in \mathcal{P}}g_{\mu}p_{(\mu,\emptyset)}\right]=\prod_{i=1}^{l(\lambda)}\left(\sum_{\mu \in \mathcal{P}} g_{\mu}p_{(\lambda_i\mu,\emptyset)}\right) 
$$
We deduce eq.(\ref{plethysmsn-bn-eq}).
\end{proof}

\section{Compositional formulas}
\label{compositional-formulas-paragraph}

In this section, we will focus on the following. Put  $\Lambda[x]=\Lambda(X) \otimes_{\mathbb{Q}}\mathbb{Q}[x]$. As remarked in  Section \ref{paragraph-plethysm-series}, this ring has a natural $\lambda$-ring structure. Assume now to have two formal power series $$H(t)=\sum_{n \geq 1} H_n t^n $$ and $$G(t)=\sum_{n \geq 1} G_n t^n $$ with $H_n,G_n \in \Lambda_n$ for each $n \geq 1$. We define the composition as follows $$H(G(t))\coloneqq\sum_{n \geq 1} H_n[G(t)] .$$

We want to give an explicit description of the coefficients of this composition, which is the content of Theorem \ref{compositionalformulaA} below. In order to do this, we will need the following preliminaries.
 
 Given $k \in \mathbb{N}_{\geq 1}$, let $\Lambda(k) \coloneqq \Lambda(X) \otimes_{\Q} \cdots \otimes_{\Q} \Lambda(X)$. Notice that in this case we have a multiplication map $m_k\colon \Lambda(k) \to \Lambda$.
 
 We define the $k$-plethysm $f[g]$ of two elements $f,g \in \Lambda(k)$ as follows. If $\displaystyle f=\sum_{\lambda=(\lambda^1,\dots,\lambda^k) \in \mathcal{P}^k}p_{\lambda^1}(X)\otimes \cdots \otimes p_{\lambda^k}(X)f_{\lambda}$, we define $$f[g] \coloneqq \sum_{\lambda=(\lambda^1,\dots,\lambda^k) \in \mathcal{P}^k}f_{\lambda} \psi_{\lambda^1} \otimes \cdots \otimes \psi_{\lambda^k}(g) .$$

From this definition, we see that if $f=f_1 \otimes \cdots \otimes f_k$ and $g=g_1 \otimes \cdots \otimes g_k$, we have $$f[g]=f_1[g_1] \otimes \cdots \otimes f_k[g_k] .$$

In particular, we have that $$\ch \Ind_{S_{j_1} \wr S_{m_1} \times \cdots \times S_{j_k} \wr S_{m_k}}^{S_{n_1} \times \cdots \times S_{n_k}}((W_1^{\boxtimes m_1} \otimes V_1) \boxtimes \cdots \boxtimes (W_k^{\boxtimes m_k} \otimes V_k))=$$ $$=\ch(V_1 \boxtimes \cdots \boxtimes V_k)[\ch(W_1 \boxtimes \cdots \boxtimes W_k)] .$$

By linearity, we deduce that, for any representation $V$ of $S_{m_1} \times \cdots \times S_{m_k}$, we have that \begin{equation}
    \ch \Ind_{S_{j_1} \wr S_{m_1} \times \cdots \times S_{j_k} \wr S_{m_k}}^{S_{n_1} \times \cdots \times S_{n_k}}((W_1^{\boxtimes m_1}  \boxtimes \cdots \boxtimes W_k^{\boxtimes m_k}) \otimes V)=\ch(V)[\ch(W_1) \otimes \cdots \otimes \ch(W_k)]\label{kplethysm}
\end{equation}

Notice now that for any integer $n \geq 1$ and $n_1,\dots,n_k$ such that $n_1+\cdots +n_k=n$, we have the restriction map $\Res_{S_{n_1} \times \cdots S_{n_k}}^{S_n}$, which induces thus a restriction map $\Res_{\Lambda(k)}^{\Lambda_n}\colon \Lambda_n \to \Lambda_{n_1} \otimes \cdots \otimes \Lambda_{n_k}$. For any $\mu \in \mathcal{P}_n$, we have that \begin{equation}\label{restriction}\Res_{\Lambda(k)}^{\Lambda_n}(p_{\mu}z_{\mu}^{-1})=\sum_{\substack{(\mu^1,\dots,\mu^k) \in \mathcal{P}_{n_1} \times \cdots \times \mathcal{P}_{n_k} \\ \mu^1 \cup \cdots \cup \mu^k=\mu }}p_{\mu^1}z_{\mu^1}^{-1} \otimes \cdots \otimes p_{\mu^k}z_{\mu^k}^{-1} .\end{equation}

Notice that we have $\sum_{j \in \supp(\lambda)} m_{j,\lambda}=l(\lambda)$ and thus a restriction morphism $\Res_{\Lambda(|\supp(\lambda)|)}^{\Lambda_{l(\lambda)}}$.

\begin{theorem}
\label{compositionalformulaA}
The coefficient of $t^n$ in the power series $H(G(t))$ is given as follows.

\begin{equation}
    \sum_{\lambda \in \mathcal{P}_n}m_{l(\lambda)}\left(\Res_{\Lambda(|\supp(\lambda)|)}^{\Lambda_{l(\lambda)}}(H_{l(\lambda)})[\prod_{j \in \supp(\lambda)} G_j]\right).
\end{equation}
\end{theorem}

\begin{proof}
For each $m \in \mathbb{N}_{\geq 1}$, write $H_m=\displaystyle \sum_{\mu \in  \mathcal{P}_m}H_{\mu}z_{\mu}^{-1}p_{\mu}$. If $\mu=(\mu_1,\dots,\mu_{l(\mu)})$, we have $$p_{\mu}[G(t)]=\prod_{i=1}^{l(\mu)}\left(\sum_{a_i \geq 1}p_{\mu_i}[G_{a_i}]t^{a_i\mu_i}\right) .$$

Let $A_n=\{(\mu,a=(a_1,\dots,a_{l(\mu)}) \in \mathcal{P} \times \mathbb{N}^{l(\mu)} \ | \ \sum_i \mu_i a_i=n\}$. The coefficient of $t^n$ in $H(G(t))$ is thus given by $$\sum_{(\mu,a) \in A_n}H_{\mu}z_{\mu}^{-1}\prod_{i=1}^{l(\mu)}p_{\mu_i}[G_{a_i}] .$$

We have a map $f\colon A_n \to \mathcal{P}_n$ defined as follows. The element $f((\mu,a))$ is the partition $\lambda \in \mathcal{P}_n$ such that $\supp(\lambda)=\{a_1,\dots,a_{l(\mu)}\}$ and $m_{\lambda,a_j}=\displaystyle\sum_{\substack{h \text{ s.t.}\\ a_h=a_j}}\mu_h$. Notice that $l(f(\mu,a))=|\mu|$.

Fix now $\lambda \in \mathcal{P}_n$. We compute now $$\sum_{\substack{(\mu,a) \in A_n\\ f(\mu,a)=\lambda}}H_{\mu}z_{\mu}^{-1}\prod_{i=1}^{l(\mu)}p_{\mu_i}[G_{a_i}] .$$
 Let $g \colon \{(\mu,a) \in A_n \ | \ f(\mu,a)=\lambda\} \to \mathcal{P}_{l(\lambda)}$ be the function $g(\mu,a)=\mu$. Put $k=m(\lambda)$, let $\supp(\lambda)=(b_1,\dots,b_k)$. For any $\mu \in \mathcal{P}_{l(\lambda)}$, we defined the following map $$\psi \colon g^{-1}(\mu) \to \{(\mu^1,\dots,\mu^k) \in \mathcal{P}_{m_{\lambda,b_1}} \times \cdots \times \mathcal{P}_{m_{\lambda,b_k}} \ | \ \mu^1 \cup \cdots  \cup \mu^k=\mu\} $$  as follows. Given an element $(\mu,a) \in g^{-1}(\mu)$ we associate to it the multipartition $\psi(\mu,a)=(\psi(\mu,a)^1,\dots,
\psi(\mu,a)^k)$ where $\psi(\mu,a)^j$ is the partition whose entries are given by the integers $\mu_h$ with $a_h=b_j$. Notice that, if $\psi(\mu,a)=(\mu^1,\dots,\mu^k)$, we have $$\prod_{i=1}^{l(\mu)}p_{\mu_i}[G_{a_i}]=\prod_{j=1}^k p_{\mu^j}[G_{b_j}] .$$

 In particular, we have that \begin{align}&\sum_{\substack{(\mu,a) \in A_n\\ f(\mu,a)=\lambda}}H_{\mu}z_{\mu}^{-1}\prod_{i=1}^{l(\mu)}p_{\mu_i}[G_{a_i}]= \nonumber \\ &=\sum_{\mu \in \mathcal{P}_{l(\lambda)}}H_{\mu}z_{\mu}^{-1} \sum_{\substack{(\mu^1,\dots,\mu^k) \in \mathcal{P}_{m_{\lambda,b_1}} \times \cdots \times \mathcal{P}_{m_{\lambda,b_k}} \\ \mu^1 \cup \cdots  \cup \mu^k=\mu}}|\psi^{-1}(\mu,a)| \prod_{j=1}^k p_{\mu^j}[G_{b_j}] .\label{rewriting1}\end{align}

Assume that $\supp(\mu)=\{c_1,\dots,c_t\}$. We have that \begin{equation}
    |\psi^{-1}(\mu^1,\dots,\mu^k)|=\prod_{i=1}^t \dfrac{(m_{\mu,c_i})!}{\prod_{j=1}^k (m_{\mu^j,c_i})! }=\dfrac{z_{\mu}}{z_{\mu^1} \cdots z_{\mu^k}}.
\end{equation}

From eq.(\ref{rewriting1}) and eq.(\ref{restriction}), we deduce $$\sum_{\substack{(\mu,a) \in A_n\\ f(\mu,a)=\lambda}}H_{\mu}z_{\mu}^{-1}\prod_{i=1}^{l(\mu)}p_{\mu_i}[G_{a_i}]=m_{l(\lambda)}\left(\Res_{\Lambda(|\supp(\lambda)|)}^{\Lambda_{l(\lambda)}}(H_{l(\lambda)})[\prod_{j \in \supp(\lambda)} G_j]\right)$$ and thus the Theorem.
\end{proof}

\begin{remark}
Assume that we have $H_n=H_{(1^n)}p_{(1^n)}/n!$ and $G_n=G_{(1^n)}p_{(1^n)}/n!$ for any $n$, and we put $\widetilde{H}_n=H_{(1^n)}$ and $\widetilde{G}_n=G_{(1^n)}$. For any $\lambda \in \mathcal{P}$, we have    \begin{align*}&\Res_{\Lambda(|\supp(\lambda)|)}^{\Lambda_{l(\lambda)}}(H_{l(\lambda)})[\prod_{j \in \supp(\lambda)} G_j]=\\ &\Res_{\Lambda(|\supp(\lambda)|)}^{\Lambda_{l(\lambda)}}\left(\dfrac{\widetilde{H}_{l(\lambda)}}{l(\lambda)!}p_{(1^{l(\lambda)})}\right)\Bigg[\prod_{j \in \supp(\lambda)}(\widetilde{G}_jp_{(1^j)}/j!)\Bigg]=\\ &\widetilde{H}_{l(\lambda)}\prod_{j \in \supp(\lambda)}\dfrac{\widetilde{G}^{m_{j,\lambda}}_j}{j!^{m_{j,\lambda}}}p_{(1^m)}
\end{align*}
where the last equality is a consequence of eq.(\ref{restriction}). The coefficient of degree $m$ of $H(G(t))$ is thus \begin{align*}
    &\sum_{\lambda \in \mathcal{P}_m}\widetilde{H}_{l(\lambda)}\prod_{j \in \supp(\lambda)}\dfrac{\widetilde{G}^{m_{j,\lambda}}_j}{m_{j,\lambda}!j!^{m_{j,\lambda}}}p_{(1^m)}=\dfrac{p_{(1^m)}}{m!}\sum_k \tilde{H}_{k}\sum_{{\substack{\sigma \in \Pi_m \\ l(\sigma)=k }}}\prod_{i=1}^k\widetilde{G}_{|\sigma_i|},
\end{align*}
where $\Pi_m$ is the set of set partitions of the set $\{1,\dots,m\}$. We recover thus the usual Faa di Bruno's formula for the composition of formal power series, see for instance \cite{stanley}*{Theorem 5.1.4}.
\end{remark}

\begin{remark}
\label{paragraph-equiv-bell}
Let $F(t) \in \Lambda[[t]]$ be such that $F_0=0$ and $F_1=p_1t$. Using Theorem \ref{compositionalformulaA}, it is possible to show that it exists a unique series $G(t)$ such that $F[G]=p_1t$, i.e. a plethystic inverse. An equivalent of Bell's inversion formula exists for such an inverse and can be found in multiple works, where the coefficients of $G$ are expressed in terms of the coefficients of $F$ through an iterative process indexed by rooted trees, see for instance \cite{Species},\cite{Yang}.    
\end{remark}

\subsection{Type \texorpdfstring{$B$}{B} compositional formula}

In this section we want to generalize Theorem \ref{compositionalformulaA} to the case of plethysm of bisymmetric functions. As remarked in  Section \ref{paragraph-plethysm-series}, the ring $\Lambda^{(2)}[[t]]$ has a natural $\lambda$-ring structure. Given any $f \in 1+t\Lambda^{(2)}[[t]]$ and any $g \in \Lambda^{(2)}[[t]]$, we put $$f^{\langle g \rangle_B} \coloneqq \Exp(g \Logg(f)^{\even}) .$$

Given  $G(t)=\sum_{n \geq 1}G_n t^n \in \Lambda^{(2)}(t)$, put $G^{\even}(t)=\sum_{n \geq 1} G_n^{\even}t^n$ and $G^{\odd}(t)=\sum_{n \geq 1}G_n^{\odd}t^n$. For $f_1 \otimes f_2 \in \Lambda^{(2)}$, we define $f_1 \otimes f_2[G(t)]=f_1[G^{\even}(t)]f_2[G^{\odd}(t)]$ and we extend it by linearity to define $f[G(t)]$ for every $f \in \Lambda^{(2)}$.

Assume  now to have two formal power series $$H(t)=\sum_{n \geq 1} H_n t^n $$ and $$G(t)=\sum_{n \geq 1} G_n t^n $$ with $H_n,G_n \in \Lambda^{(2)}_n$ for each $n \geq 1$. We define the composition as follows $$H(G(t))=\sum_{n \geq 1} H_n[G(t)] .$$ We want to describe the coefficients of this composition. As before, we will need the following preliminaries.

In what follows, for any $k \in \N$, we define $\Lambda^{(2)}(k)$ similarly as above. Notice that in this case we have  a multiplication map $m_k\colon \Lambda^{(2)}(k) \to \Lambda$. We define the $k$-plethysm $f[g]$ of two elements $f,g \in \Lambda^{(2)}(k)$ by analogy with what was done for the case of symmetric functions.

In a similar way to the case of symmetric functions above, we deduce that, for any representation $V$ of $B_{m_1} \times \cdots \times B_{m_k}$, we have that \begin{align*}
    &\ch \Ind_{B_{j_1} \wr S_{m_1} \times \cdots \times B_{j_k} \wr S_{m_k}}^{B_{n_1} \times \cdots \times B_{n_k}}((W_1^{\boxtimes m_1}  \boxtimes \cdots \boxtimes W_k^{\boxtimes m_k}) \otimes V)= \nonumber\\ &\ch(V)[\ch(W_1)[X+Y] \otimes \cdots \otimes \ch(W_k)[X+Y]] 
\end{align*}

Notice now that for any integer $n \geq 1$ and $n_1,\dots,n_k$ such that $n_1+\cdots +n_k=n$, we have the restriction map $\Res_{B_{n_1} \times \cdots B_{n_k}}^{B_n}$, which induces thus a restriction map $\Res_{\Lambda^{(2)}(k)}^{\Lambda^{(2)}_n}\colon \Lambda^{(2)}_n \to \Lambda^{(2)}_{n_1} \otimes \cdots \otimes \Lambda^{(2)}_{n_k} \subseteq \Lambda^{(2)}(k)$. 

For any $\mu=(\mu^{(0)},\mu^{(1)}) \in \mathcal{P}_n$, we have that \begin{equation*}
\Res_{\Lambda^{(2)}(k)}^{\Lambda^{(2)}_n}(p_{\mu}z_{\mu}^{-1})=\sum_{\substack{(\mu^1,\dots,\mu^k) \in \mathcal{P}^{(2)}_{n_1} \times \cdots \times \mathcal{P}^{(2)}_{n_k} \\ \mu^1 \cup \cdots \cup \mu^k=\mu }}p_{\mu^1}z_{\mu^1}^{-1} \otimes \cdots \otimes p_{\mu^k}z_{\mu^k}^{-1} .
\end{equation*}

Notice that similarly as above, for any $\lambda \in \mathcal{P}$, we have a restriction morphism $\Res_{\Lambda^{(2)}(m(\lambda))}^{\Lambda^{(2)}_{l(\lambda)}}$. 

We have the following compositional formula.

\begin{theorem}
\label{compositional-formulatypeB}
The coefficient of $t^n$ in the composition $H(G(t))$ is:
\begin{equation}
    \sum_{\lambda \in \mathcal{P}_n}m_{l(\lambda)}\left(\Res^{\Lambda^{(2)}_{l(\lambda)}}_{\Lambda^{(2)}(|\supp(\lambda)|)}(H_{l(\lambda)})[\prod_{j \in \supp(\lambda)}G_j]\right).
\end{equation}
\end{theorem}

\begin{proof}
For each $m \geq 1$, write $H_m=\displaystyle\sum_{\mu \in \mathcal{P}_m^{(2)}}H_{\mu}z_{\mu}^{-1}p_{\mu}$. we have $$p_{\mu}[G(t)]=\prod_{i=1}^{l(\mu_i^{(0)})}\left(\sum_{a^{(0)}_i \geq 1}p_{((\mu^{(0)}_i),\emptyset)}[G_{a_i^{(0)}}]t^{\mu_i^{(0)}a_i^{(0)}}\right)\prod_{j=1}^{l(\mu_i^{(1)})}\left(\sum_{a^{(1)}_j \geq 1}p_{(\emptyset,(\mu^{(1)}_j))}[G_{a_j^{(1)}}]t^{\mu_j^{(1)}a_j^{(1)}}\right) .$$

Let $A_n^{(2)}=\{(\mu,a=(a^{(0)}_1,\dots,a^{(0)}_{l(\mu^{(0)})},a^{(1)}_1,\dots,a^{(1)}_{l(\mu^{(1)})}) \in \mathcal{P}^{(2)} \times \N^{l(\mu)} \ | \ \sum_{j,\epsilon} \mu^{(\epsilon)}_j a_j^{(\epsilon)}=n\}$. We have thus that the coefficient of $t^n$ in $H[G(t)]$ is \begin{equation}
  \sum_{(\mu,a) \in A^{(2)}_n}H_{\mu}z_{\mu^{-1}}\prod_{i=1}^{l(\mu_i^{(0)})}p_{((\mu^{(0)}_i),\emptyset)}[G_{a_i^{(0)}}]\prod_{j=1}^{l(\mu_i^{(1)})}p_{(\emptyset,(\mu^{(1)}_j))}[G_{a_j^{(1)}}] .
\end{equation}
We have a map $f\colon A^{(2)}_n \to \mathcal{P}_n$ as follows. Given an element $(\mu,a) \in A_n^{(2)}$, let $$\{a^{(0)}_1,\dots,a^{(0)}_{l(\mu^{(0)})},a^{(1)}_1,\dots,a^{(1)}_{l(\mu^{(1)})}\}=\{b_1,\dots,b_k\}$$ and the partition $\lambda=f(\mu,a)$ is given as follows. The partition $\lambda$ is the partition with support $\{b_1,\dots,b_k\}$ and where $\displaystyle m_{\lambda,b_j}=\sum_{\substack{h,\epsilon  \text{ s. t }\\ a^{\epsilon}_h=b_j}}\mu_h^{(\epsilon)}$.
The proof is now the analog of that of Theorem \ref{compositionalformulaA}.
\end{proof}

\section{Equivariant incidence algebras and  hyperplane arrangements for classical root systems}
\label{sec:typeAB}

Given a topological space $X$ with an action of a finite group $W$ on it, each cohomology group $H^i(X,\Q),H^i_c(X,\Q)$ carries a natural action of $H$ on it. We define thus  $$P_{X}^W(x) \coloneqq \sum_{i \geq 0} H^i(X,\Q) x^i \in \Vrep(W)[x] $$ $$P_{X,c}^W(x) \coloneqq \sum_{i \geq 0} H_c^i(X,\Q) x^i \in \Vrep(W)[x] $$

\subsection{The case of the partition lattice}

We apply the results of Section \ref{chapter-incidence-algebra} to the case of the partition lattices. Fix $n \in \mathbb{N}$ and let $\mathcal{L}=\Pi_n$ be the partition lattice of $n$, i.e. the elements of $\Pi_n$ are partitions of $\{1,\dots,n\}$ with $A \leq B \in \Pi_n$ if and only if $A$ is a refinement of $B$. In this case, $\hat{0}=\{\{1\},\dots,\{n\}\}$ and $\hat{1}=\{1,\dots,n\}$.

Notice that for any $\sigma \in \Pi_n$, if its associated partition is $\lambda \in \mathcal{P}_n$, we have that  \begin{align}\label{isomorphism_stabilizer_Alattice}
    &[\sigma,\hat{1}]\cong \Pi_{l(\lambda)}\\
&[\hat{0},\sigma] \cong \prod_{j \in \supp(\lambda)}\Pi_j^{m_{j,\lambda}} .\end{align}

We choose $\mathcal{G}=\mathcal{G}_{\min}$  the minimal building set. For the partition lattice $\Pi_n$, an element $X \in \Pi_n$ is irreducible if and only if it has exactly $1$ block of size greater or equal than $2$. 

We look at the group $W=S_n$ with the canonical action by permutation of the elements of $\{1,\dots,n\}$. Notice that, for any $\sigma \in \Pi_n$, if its associated partition is $\lambda \in \mathcal{P}_n$, we have that \begin{equation*}
S_{\sigma}= \prod_{j \in \supp(\lambda)}S_j \wr S_{m_{\lambda,j}} .\end{equation*}   

Reduced equivariant characteristic polynomials and Chow polynomials in this case have the following geometric interpretation in terms of hyperplane arrangement and moduli spaces of pointed curves. For any $n \in \N_{>0}$, let $V_n=\C^n/\Span(v)$ where $v=(1,\dots,1)$. Let $\mathcal{A}_n$ be the hyperplane arrangement of type $A_{n-1}$ inside $V_n$ i.e. the collection of hyperplanes $H_{i,j}=\{(x_1,\dots,x_n) \in V_n \ | \ x_i=x_j\}$.
We denote by $M(\mathcal{A}_n)$ the complement, i.e. $\displaystyle M(\mathcal{A}_n)=V_n\setminus \bigcup_{i,j}H_{i,j}$ and $\mathbb{P}(M(\mathcal{A}_n))$ its projectivization.  The space $\mathbb{P}(M(\mathcal{A}_n))$ admits a (minimal) wonderful compactification $\overline{\mathbb{P}(M(\mathcal{A}_n))}_{\min}$, due to the results of \cite{DeConciniProcesi}.  In this case, we have 
\[(\overline{\chi}_{\Pi_n}^{\mathcal{G}_{min}})_{[\hat{0},\hat{1}]}(x)=x^{n}P_{\mathbb{P}(M(\mathcal{A}_n))}^{S_n}(-1/x)=x^n\dfrac{P_{M(\mathcal{A}_n)}(-1/x)}{x-1}\]
and $(H_{\Pi_n}^{\mathcal{G}_{min}})_{[\hat{0},\hat{1}]}(x^2)=P_{\overline{\mathbb{P}(M(\mathcal{A}_n))}_{\min}}^{S_n}(x)$.

The spaces associated to hyperplane arrangements of type $A$ have the following modular interpretation. The variety $\mathbb{P}(M(\mathcal{A}_n))$ is isomorphic to the moduli space of smooth curves of genus $0$ with $n+1$-marked points $\mathcal{M}_{0,n+1}$ and its minimal compactification $\overline{\mathbb{P}(M(\mathcal{A}_n))}_{\min}$ is isomorphic to the Deligne-Knudsen-Mumford space $\overline{\mathcal{M}}_{0,n+1}$ of stable curves of genus $0$ with $n+1$-marked points. We can resume the discussion above in the following:
We have that
\begin{equation*}
    (H_{\Pi_n}^{\mathcal{G}_{min}})_{[\hat{0},\hat{1}]}(x^2)=P_{\overline{\mathcal{M}}_{0,n+1}}^{S_n}(x) 
\end{equation*}
and 
\begin{equation*}
(\overline{\chi}_{\Pi_n}^{\mathcal{G}_{min}})_{[\hat{0},\hat{1}]}(x)=x^{n} P_{\mathcal{M}_{0,n+1}}^{S_n}(-1/x).
\end{equation*}

For any $n \in \mathbb{N}_{>0}$, we put $H^A_n \coloneqq \ch((H_{\Pi_n}^{\mathcal{G}_{min}})_{[\hat{0},\hat{1}]}) \in \Lambda(X)[x]$ and $\overline{\chi}^A_n \coloneqq \ch((\overline{\chi}_{\Pi_n}^{\mathcal{G}_{min}})_{[\hat{0},\hat{1}]})$.
Moreover, we put $$H^A(t) \coloneqq \sum_{n \geq 1} H^A_n t^n$$ and $$\overline{\chi}^A(t)\coloneqq \sum_{n \geq 1} \overline{\chi}^A_nt^n .$$

It is proved in \cite{GetzlerOperads}*{Proposition 5.5} that we have 
\begin{equation}
    \label{equation_chitypeA}
    \overline{\chi}^A(t)=\dfrac{\prod_{n \geq 1}(1+p_nt^n)^{T^A_n(x)}-1-p_1xt}{x(x-1)}-p_1t
\end{equation}
where $\displaystyle T^A_n(x)=\dfrac{1}{n}\sum_{d | n}\mu\left(\dfrac{n}{d}\right)x^d$.

\begin{theorem}\label{thm:Getzler_main}
We have an equality $$H^A(-\overline{\chi}^A)=p_1t ,$$ i.e. $H^A$, $-\overline{\chi}^A$ are plethystic inverses one to each other.    
\end{theorem}

\begin{proof}
Notice that we have 
$$( \overline{\chi}^{\mathcal{G}_{min}}_{\Pi_n} \ast H_{\Pi_n}^{\mathcal{G}_{min}})_{[\hat{0},\hat{1}]}=\sum_{\sigma \in \Pi_n} \dfrac{n!}{|S_{\sigma}|} \Ind_{S_{\sigma}}^{S_n}( (\overline{\chi}^{\mathcal{G}_{min}}_{\Pi_n})_{[\hat{0},\sigma]} \otimes (H_{\Pi_n}^{\mathcal{G}_{min}})_{[\sigma,\hat{1}]}) .$$

We put $$( \overline{\chi}^{\mathcal{G}_{\min}}_{\Pi_n} \ast H_{\Pi_n}^{\mathcal{G}_{\min}})_{[\hat{0},\hat{1}],\sigma}=\Ind_{S_{\sigma}}^{S_n}( (\overline{\chi}^{\mathcal{G}_{\min}}_{\Pi_n})_{[\hat{0},\sigma]} \otimes (H_{\Pi_n}^{\mathcal{G}_{\min}})_{[\sigma,\hat{1}]})$$ and we describe now each of these terms.

Notice that, through the identification of eq.(\ref{isomorphism_stabilizer_Alattice}), for the poset $\Pi_{l(\lambda)}$ we have that $(\mathcal{G}_{\min})_{\sigma}^{\hat{1}}=\mathcal{G}_{\min}$ . Moreover, in the stabilizer $S_{\sigma}$, each group $S_j$ acts trivially on the lattice $[\sigma,\hat{1}]$. We have an inclusion $\displaystyle\prod_{j \in \supp(\lambda)}S_{m_{\lambda,j}} \subseteq S_{l(\sigma)}$ and thus for each $V$ representation of $S_{l(\sigma)}$ the restriction $\Res^{S_{l(\sigma)}}_{\prod_{j \in \supp(\lambda)}S_{m_{\lambda,j}}}V $ can be endowed with a natural structure of a $S_{\sigma}$-representation. We have thus an equality in $\Vrep(S_{\sigma})[x]$:
\begin{equation*}
(H_{\Pi_n}^{\mathcal{G}_{min}})_{[\sigma,\hat{1}]} \cong  \Res^{S_{l(\sigma)}}_{\prod_{j \in \supp(\lambda)}S_{m_{\lambda,j}}}(H_{\Pi_{l(\sigma)}}^{\mathcal{G}_{min}})_{[\hat{0},\hat{1}]}  
\end{equation*}

Notice that, through the identification of eq.(\ref{isomorphism_stabilizer_Alattice}), $\max \G$ is in bijection with the blocks of $\sigma$ and it has thus size $l(\lambda)$. Moreover,  each group $S_j$ acts trivially on $\max \G$ and $S_{m_{j,\lambda}}$ acts by permutation of the blocks of the same size. Thus $\sgn_{\max \G}=\prod_j \sgn_{S_{m_{j,\lambda}}}$ seen as a representation of $S_{\sigma}$. Putting all together, we have an equality

\begin{align*}
 &(\overline{\chi}^{\mathcal{G}_{min}}_{\Pi_n})_{[\hat{0},\sigma]}= \\ &=(-1)^{l(\lambda)-1}\prod_{j \in \supp(\lambda)}(\overline{\chi}^{\mathcal{G}_{min}}_{\Pi_j})_{[\hat{0},\hat{1}]}^{\boxtimes m_{j,\lambda}} \otimes \bigotimes_{j \in \supp(\lambda)}\sgn_{S_{m_{j,\lambda}}}   
\end{align*}

Put $\overline{S_{\sigma}}=\displaystyle \prod_{j \in \supp(\lambda)} S_{jm_{j,\lambda}}$. We deduce that we have \begin{align}
    &( \overline{\chi}^{\mathcal{G}_{min}}_{\Pi_n} \ast H_{\Pi_n}^{\mathcal{G}_{min}})_{[\hat{0},\hat{1}],\sigma}=\nonumber\\ &\Ind^{S_n}_{S_{\sigma}} \prod_{j \in \supp(\lambda)}(\overline{\chi}^{\mathcal{G}_{min}}_{\Pi_j})_{[\hat{0},\hat{1}]}^{\boxtimes m_{j,\lambda}} \otimes \bigotimes_{j \in \supp(\lambda)}\sgn_{S_{m_{j,\lambda}}} \Res^{S_{l(\sigma)}}_{\prod_{j \in \supp(\lambda)}S_{m_{\lambda,j}}}(H_{\Pi_{l(\sigma)}}^{\mathcal{G}_{min}})_{[\hat{0},\hat{1}]}=\nonumber \\
&=(-1)^{l(\lambda)}\Ind^{S_n}_{\overline{S_{\sigma}}} \Ind^{\overline{S_{\sigma}}}_{S_{\sigma}} \prod_{j \in \supp(\lambda)}(\overline{\chi}^{\mathcal{G}_{min}}_{\Pi_j})_{[\hat{0},\hat{1}]}^{\boxtimes m_{j,\lambda}} \otimes \Res^{S_{l(\sigma)}}_{\prod_{j \in \supp(\lambda)}S_{m_{\lambda,j}}}(\sgn \otimes H_{\Pi_{l(\sigma)}}^{\mathcal{G}_{min}})_{[\hat{0},\hat{1}]}\label{plehtysm-label}    
\end{align}

From eq.(\ref{plethysm-minus}) and eq.(\ref{kplethysm}) and eq.(\ref{plehtysm-label}) above, we see that we have thus an equality \begin{equation}
\label{compositional-latticeA}
    \ch(( \overline{\chi}^{\mathcal{G}_{min}}_{\Pi_n} \ast H_{\Pi_n}^{\mathcal{G}_{min}})_{[\hat{0},\hat{1}]})=-\sum_{\lambda \in \mathcal{P}_n}m_{l(\lambda)}\left(\Res_{\Lambda(|\supp(\lambda)|)}^{\Lambda_{l(\lambda)}}(H^A_{l(\lambda)})[\prod_{j \in \supp(\lambda)}-\overline{\chi}^A_j]\right)
\end{equation}

From Theorem \ref{theorem_inversion_Chow}, we have that \begin{equation}
\label{equivariant-incidence-typeA}
    ( \overline{\chi}^{\mathcal{G}_{min}}_{\Pi_n} \ast H_{\Pi_n}^{\mathcal{G}_{min}})_{[\hat{0},\hat{1}]}=\begin{cases}
        0 \text{ if } n \geq 2\\
        -p_1 \text{ if } n=1
    \end{cases}
\end{equation}

Finally, from eq.(\ref{equivariant-incidence-typeA}) and eq.(\ref{compositional-latticeA}) and Theorem \ref{compositionalformulaA}, we deduce the wanted result.    
\end{proof}

This result was already shown in  \cite{GetzlerOperads}, where the author used operads and the fact that, for each $n$, the representation of $S_n$ on the cohomology of $H^{\bullet}(\overline{\mathcal{M}}_{0,n+1})$ is the restriction of the $S_{n+1}$ representation coming from the permutation of the $n+1$ marked points. 

\begin{remark}
\label{inverted-formula-A}
The set of series $F \in \Lambda[x][[t]]$ such that $F_1=p_1$ is a a group with respect to the plethysm operation. We deduce thus that we have also
\begin{equation}
    -\overline{\chi}^A(H^A(t))=p_1t
    \label{remark-inversion-equation}
\end{equation}
Similarly, eq.(\ref{remark-inversion-equation}) can be shown by the same line of reasoning used in the proof above, using that, for any $n$, we have \begin{align*}
    ( H^{\mathcal{G}_{min}}_{\Pi_n} \ast \overline{\chi}_{\Pi_n}^{\mathcal{G}_{min}})_{[\hat{0},\hat{1}]}=\begin{cases}
        0 \text{ if } n \geq 2\\
        -p_1 \text{ if } n=1
    \end{cases}
\end{align*} i.e. switching the role of $\overline{\chi}$ and $H$ in our proof. As remarked in \cite{getzler-mhs}*{Theorem 2.7}, we have that   $$\overline{\chi}^A(t)=\dfrac{(1+p_1t)^{\langle x \rangle}-1-p_1xt}{x(x-1)}-p_1t .$$
Thus $$\dfrac{(1+H_A)^{\langle x \rangle}-1-H_Ax}{x(x-1)}-H_A=-p_1t $$ which is equivalent to 
\begin{equation}\label{manin-equivariant}(1+H_A)^{\langle x \rangle}-x^2H_A=1-p_1tx(x-1) \end{equation}
that is an equivariant generalization of the functional identity \cite{manin}*{Theorem 0.3.1}.
\end{remark}

\subsection{The case of the lattice of signed partitions}

Let $\Pi^B_n$ be the lattice of signed partitions. We quickly recall its definition here. Let $\langle n \rangle$ be the set $\{1,2,\dots,n,\overline{1},\dots,\overline{n}\}$. This set has a natural involution $i \to \overline{i}$, where we set $\overline{\overline{i}}=i$. For each $B \subseteq \langle n \rangle$, we set $\overline{B}=\{\overline{i} \ | \ i \in B\}$. A signed partition is a partition $\pi$ of the set $\langle n \rangle $
such that for each block $B \in \pi$, we have that $\overline{B} \in \pi$ and there is at most one block $Z$ such that $\overline{Z}=Z$. We put $\pi(0)=|Z|$. Recall that we can have $Z=\emptyset$, i.e. $\pi(0)=0$. We have that $\pi\setminus \{Z\}$ determines a partition of $n-\pi(0)/2$, which we denote by $\sigma$. We denote by $\lambda$ the corresponding element of $\mathcal{P}.$ In this case too, we say that $\pi \leq \pi'$ if and only if $\pi$ is a refinement of $\pi'$ and we work with the minimal building set $\G=\G_{\min}$. We have that $\pi \in \G_{\min}$ if and only if either $\pi(0)>0$ and $\sigma$ has only blocks of size $1$ or $\pi(0)=0$ and $\sigma$ has exactly $1$ block of size greater than $1$.

Notice that, for any $\pi \in \Pi^B_n$, we have that 
\begin{align} &[\hat{0},\pi] \cong \Pi^B_{\pi(0)/2} \times \prod_{j \in \supp(\lambda)}\Pi^{m_{j,\lambda}}_{j} \\ &[\pi,\hat{1}] \cong \Pi^B_{l(\lambda)}. \label{lattice-b-interval}
\end{align}

The group $B_n$ naturally acts on $\langle n \rangle$, where $S_n$ acts by permutation of the elements $\{1,\dots,n\}$ and $\Z_2$ acts through the involution $i \to \overline{i}.$ This action gives an action of $B_n$ on $\Pi^B_n$. Notice that, for each $\pi \in \Pi^B_n$, the pair $(\pi(0),\lambda)$ uniquely determines $\pi$ up to the action of the group $B_n$. Moreover, we have that $$B_{\pi}=\displaystyle B_{\pi(0)/2} \times \prod_{j \in \supp(\lambda)}S_j \wr B_{m_{j,\lambda}} .$$

For any $n$, similarly as above, we put $H^B_n \coloneqq \ch((H_{\Pi^B_n}^{\mathcal{G}_{min}})_{[\hat{0},\hat{1}]}) \in \Lambda^{(2)}[x]$ and $\overline{\chi}^B_n \coloneqq \ch((\overline{\chi}_{\Pi^B_n}^{\mathcal{G}_{min}})_{[\hat{0},\hat{1}]})$. These polynomials have the following geometric interpretation in this case.

For any $n$, let $\mathcal{B}_n$ be the hyperplane arrangement of type $B_n$, i.e. the collection of hyperplanes of $\C^n$ given by $Y_i=\{(x_1,\dots,x_n) \in \C^n  \ | \ x_i=0\}$, $H_{i,j}^+=\{(x_1,\dots,x_n) \in \C^n \ | \ x_i=x_j\}$ and $H_{i,j}^-=\{(x_1,\dots,x_n) \in \C^n \ | \ x_i=-x_j\}$ for $i \neq j$. We denote by $M(\mathcal{B}_n)$ the corresponding complement and by $\mathbb{P}(M(\mathcal{B}_n))$ its projectivization. This variety admits a (minimal) wonderful compactification, which we denote simply by $\overline{\mathcal{M}}^B_{0,n}$ by analogy with the case of type $A$. All these varieties admit natural $B_n$-actions and we have $$H^B_n(x^2)=P^{B_n}_{\overline{\mathcal{M}}^B_{0,n}}(x) $$ and $$\overline{\chi}^B_n(x)=\dfrac{P^{B_n}_{M(\mathcal{B}_n)}(-1/x)x^n}{x-1}=x^nP^{B_n}_{\mathbb{P}(M(\mathcal{B}_n))}(-1/x) .$$

We define  the formal power series $\overline{\chi}^B(t) \in \Lambda^{(2)}[[t]]$ as $$\overline{\chi}^B(t)=-1+\sum_{n \geq 1}\overline{\chi}^B_n(x)t^n$$ and $$H^B(t)=1+\sum_{n \geq 1}H^B_n(x)t^n .$$ Notice that the choice of the constant terms is consistent with the choices of \cite{Ferroni-Pagaria-Vecchi}. We have the following.

\begin{theorem}
\label{identity-typeB}
We have the equality
\begin{equation*}
    \overline{\chi}^B H^B(-\overline{\chi}^A[X+Y])=-1.
\end{equation*}
\end{theorem}

\begin{proof}
 Notice that
$$( \overline{\chi}^{\mathcal{G}_{min}}_{\Pi^B_n} \ast H_{\Pi^B_n}^{\mathcal{G}_{min}})_{[\hat{0},\hat{1}]}=\sum_{\pi \in \Pi^B_n} \dfrac{2^nn!}{|B_{\pi}|} \Ind_{B_{\pi}}^{B_n}( (\overline{\chi}^{\mathcal{G}_{min}}_{\Pi^B_n})_{[\hat{0},\pi]} \otimes (H_{\Pi^B_n}^{\mathcal{G}_{min}})_{[\pi,\hat{1}]}) .$$

We put $( \overline{\chi}^{\mathcal{G}_{min}}_{\Pi^B_n} \ast H_{\Pi^B_n}^{\mathcal{G}_{min}})_{[\hat{0},\hat{1}],\pi}=\Ind_{B_{\pi}}^{B_n}( (\overline{\chi}^{\mathcal{G}_{min}}_{\Pi^B_n})_{[\hat{0},\pi]} \otimes (H_{\Pi^B_n}^{\mathcal{G}_{min}})_{[\pi,\hat{1}]})$ and we describe now each of these terms.
 
Notice that, through the identification of eq.(\ref{lattice-b-interval}), we have that $\G^{\hat{1}}_{\pi}$ is the minimal building of $\Pi^B_{l(\lambda)}$.
Moreover, have an inclusion $\displaystyle\prod_{j \in \supp(\lambda)}B_{m_{\lambda,j}} \subseteq B_{l(\lambda)}$ and thus for each $V$ representation of $B_{l(\lambda)}$ the restriction $\Res^{B_{l(\lambda)}}_{\prod_{j \in \supp(\lambda)}B_{m_{\lambda,j}}}V $ can be endowed with a natural structure of a $B_{\pi}$-representation. We have thus an equality in $\Vrep(B_{\pi})[x]$:
\begin{equation*}
(H_{\Pi^B_n}^{\mathcal{G}_{min}})_{[\pi,\hat{1}]} \cong  \Res^{B_{l(\lambda)}}_{\prod_{j \in \supp(\lambda)}B_{m_{\lambda,j}}}(H_{\Pi^B_{l(\sigma)}}^{\mathcal{G}_{min}})_{[\hat{0},\hat{1}]}  
\end{equation*}

Notice that, through the identification of eq.(\ref{lattice-b-interval}), $\max \G$ is in bijection with the blocks of $\sigma$ and the block $Z$ and it has thus size $l(\lambda)+1$. Moreover, $B_{\pi(0)/2}$ acts trivially on $\max \G$,  each group $S_j$ acts trivially on $\max \G$ and $B_{m_{j,\lambda}}$ acts by permutation of the blocks of $\sigma$ the same size. Thus $\sgn_{\max \G}=\prod_j \sgn^{(1)}_{B_{m_{j,\lambda}}}$ seen as a representation of $B_{\pi}$. Putting all together, we have an equality in $\Vrep(B_{\pi})[x]$

\begin{equation*}
 (\overline{\chi}^{\mathcal{G}_{min}}_{\Pi^B_n})_{[\hat{0},\pi]}=(\overline{\chi}^{\mathcal{G}_{min}}_{\Pi^B_{\pi(0)/2}})_{[\hat{0},\hat{1}]} \boxtimes  \prod_{j \in \supp(\lambda)}(\overline{\chi}^{\mathcal{G}_{min}}_{\Pi_j})_{[\hat{0},\hat{1}]}^{\boxtimes m_{j,\lambda}} \otimes_{j \in \supp(\lambda)} \sgn^{(1)}_{B_{m_{j,\lambda}}} 
\end{equation*}

Put $\overline{B}_{\pi}=\displaystyle \prod_{j \in \supp(\lambda)} B_{jm_{j,\lambda}}$. We deduce that we have 
\begin{align*}
    &( \overline{\chi}^{\mathcal{G}_{min}}_{\Pi^B_n} \ast H_{\Pi^B_n}^{\mathcal{G}_{min}})_{[\hat{0},\hat{1}],\pi}=
    \\
    &= \begin{multlined}[t] (-1)^{l(\lambda)}\Ind^{B_n}_{ B_{\pi}}(\overline{\chi}^{\mathcal{G}_{min}}_{\Pi^B_{\pi(0)/2}})_{[\hat{0},\hat{1}]} \boxtimes \prod_{j \in \supp(\lambda)}(\overline{\chi}^{\mathcal{G}_{min}}_{\Pi_j})_{[\hat{0},\hat{1}]}^{\boxtimes m_{j,\lambda}} \otimes \Res^{B_{l(\lambda)}}_{\prod_{j \in \supp(\lambda)}B_{m_{\lambda,j}}}(H_{\Pi^B_{l(\sigma)}}^{\mathcal{G}_{min}})_{[\hat{0},\hat{1}]} \otimes \\ \otimes_{j \in \supp(\lambda)} \sgn^{(1)}_{B_{m_{j,\lambda}}} 
    \end{multlined}
    \\
&= \begin{multlined}[t] (-1)^{l(\lambda)}\Ind^{B_n}_{B_{\pi(0)/2} \times \overline{B}_{\pi}} \Ind^{B_{\pi(0)/2} \times\overline{B}_{\pi}}_{B_{\pi}} (\overline{\chi}^{\mathcal{G}_{min}}_{\Pi^B_{\pi(0)/2}})_{[\hat{0},\hat{1}]} \boxtimes \prod_{j \in \supp(\lambda)}(\overline{\chi}^{\mathcal{G}_{min}}_{\Pi_j})_{[\hat{0},\hat{1}]}^{\boxtimes m_{j,\lambda}} \otimes \\
 \otimes\Res^{B_{l(\lambda)}}_{\prod_{j \in \supp(\lambda)}B_{m_{\lambda,j}}}(H^{\mathcal{G}_{min}}_{\Pi^B_{l(\sigma)}} \otimes \sgn^{(1)}_{B_{l(\lambda)}})_{[\hat{0},\hat{1}]}         
 \end{multlined}
\end{align*}

Similarly to eq.(\ref{plethysm-minus}), for any $f,g \in \Lambda^{(2)}$ with $f$ homogenous of degree $d$, we have $$\omega(f)[g]=(-1)^{d}f[-g] .$$

From Theorem \ref{thm:Ind_repr_wr}, we see that we have thus an equality \begin{equation*}
\begin{split}
    \ch\left(\prod_{j \in \supp(\lambda)}(\overline{\chi}^{\mathcal{G}_{min}}_{\Pi^B_j})_{[\hat{0},\hat{1}]}^{\boxtimes m_{j,\lambda}} \otimes \Res^{B_{l(\lambda)}}_{\prod_{j \in \supp(\lambda)}B_{m_{\lambda,j}}}(H_{\Pi^B_{l(\sigma)}}^{\mathcal{G}_{min}})_{[\hat{0},\hat{1}]}\right)= \\
    = \Res^{\Lambda^{(2)}_{l(\lambda)}}_{\Lambda^{(2)}|(\supp(\lambda))|}H^{B}_{l(\lambda)}[\prod_{j \in \supp(\lambda)}-\overline{\chi}^A_j[X+Y]],
\end{split}
\end{equation*}
from which we deduce that we have 
\begin{equation}
\label{compositional-latticeB2}
    \ch\left( \overline{\chi}^{\mathcal{G}_{min}}_{\Pi^B_n} \ast H_{\Pi^B_n}^{\mathcal{G}_{min}})_{[\hat{0},\hat{1}],\pi}\right)=\overline{\chi}^B_{\pi(0)/2}m_{l(\lambda)} \left(\Res^{\Lambda^{(2)}_{l(\lambda)}}_{\Lambda^{(2)}|(\supp(\lambda))|}H^{B}_{l(\lambda)}[\prod_{j \in \supp(\lambda)}-\overline{\chi}^A_j[X+Y]]\right)
\end{equation}

From Theorem \ref{theorem_inversion_Chow}, we have that, for every $n \geq 1$, 
\begin{equation}
\label{equivariant-incidence-typeB}
    ( \overline{\chi}^{\mathcal{G}_{min}}_{\Pi^B_n} \ast H_{\Pi^B_n}^{\mathcal{G}_{min}})_{[\hat{0},\hat{1}]}=0
\end{equation}

Finally, from eq.(\ref{equivariant-incidence-typeB}) and eq.(\ref{compositional-latticeB2}) and Theorem \ref{compositional-formulatypeB}, we deduce the wanted result.    
\end{proof}

\subsection{The case of lattice signed partitions of type \texorpdfstring{$D$}{D}}
For any $n$, let $\mathcal{D}_n$ be the hyperplane arrangement of type $D$ and rank $n$, i.e. the collection of hyperplanes of $\C^n$ given by $H^+_{i,j}=\{(x_1,\dots,x_n)\in \C^n \ | \ x_i=x_j\}$ and $H^-_{i,j}=\{(x_1,\dots,x_n)\in \C^n \ | \ x_i=-x_j\}$ for all $i \neq j$. We denote by $M(\mathcal{D}_n)$ the corresponding complement and by $\mathbb{P}(M(\mathcal{D}_n))$ its projectivization. For the latter variety, we can construct a (minimal) wonderful compactification, which by analogy with the case of type $A$, we denote by $\overline{\mathcal{M}}_{0,n}^D$. The group $B_n$ acts on $M(\mathcal{D}_n),\mathbb{P}(M(\mathcal{D}_n))$ and the wonderful model $\overline{\mathcal{M}}_{0,n}^D$ and so, by restriction, the group $D_n$ acts on these three varieties too. We want to understand the action of $B_n$ on the cohomology of these varieties. Since the restriction map  $\Res \colon \Vrep(B_n) \to \Vrep(D_n)$ is totally explicit in the basis of irreducible characters, this gives also a description of the action of $D_n$ on the cohomology of these varieties. For more details about the representation theory of the Weyl groups $D_n$, see for instance \cite{GECK}*{Chapter 5}.

We denote by $\mathcal{L}^D_n$ the intersection lattice associated to $\mathcal{D}_n$, i.e. the lattice of intersections of hyperplanes of $\mathcal{D}_n$ with the order given by reverse inclusion. 

The lattice $\mathcal{L}^D_n$ has the following description in terms of signed-partitions. We quickly recall the definitions here.  Let $\Pi^D_n \subseteq \Pi^B_n$ be the subset of signed partitions such that $|Z| \neq 2$. We have an isomorphism of geometric lattices $$\mathcal{L}^D_n \cong \Pi^D_n .$$

Notice that, for any $\pi \in \Pi^D_n$, we have an isomorphism $$[\hat{0},\pi] \cong \Pi^D_{|Z|/2} \times \prod_j \Pi_j^{m_{j,\lambda}} .$$ To describe the interval $[\pi,\hat{1}]$, we need the following notation. For any $n$ and $a \in \{0,\dots,n\}$, let $\mathcal{D}_{a,n}$ be the hyperplane arrangement of $\C^n$ consisting of hyperplanes $H^+_{i,j},H^{-}_{i,j}$ as above and the hyperplanes $H_i=\{(x_1,\dots,x_n) \ | \ x_i=0\}$ for all $i \geq a$. Notice that $\mathcal{D}_{n,n}=\mathcal{D}_n$ and $\mathcal{D}_{0,n}=\mathcal{B}_n$, i.e. the hyperplane arrangement of type $B$ in $\C^n$. Let $\mathcal{L}_{a,n}$ be the intersection lattice of the hyperplane arrangement $\mathcal{D}_{a,n}$. We have the following:

\begin{equation}[\pi,\hat{1}] \cong \begin{cases} \Pi^B_{l(\lambda)} \text{ if } |Z|>0 \\ 
                                     \mathcal{L}_{m_{1,\lambda},n}\end{cases}\label{intervalstypeD}\end{equation}

The group $B_n$ naturally acts on $\langle n \rangle$, where $S_n$ acts by permutation of the elements $\{1,\dots,n\}$ and $\Z_2$ acts through the involution $i \to \overline{i}.$ This action gives an action of $B_n$ on $\Pi^B_n$ and $\Pi^D_n$. Notice that, for each $\pi \in \Pi^B_n$, the pair $(\pi(0),\lambda)$ uniquely determines $\pi$ up to the action of the group $B_n$. Moreover, we have that $$B_{\pi}=\displaystyle B_{\pi(0)/2} \times \prod_{j \in \supp(\lambda)}S_j \wr B_{m_{j,\lambda}} .$$

For any $n$, we put $$H^D_n \coloneqq \ch((H_{\Pi^D_n}^{\mathcal{G}_{min}})_{[\hat{0},\hat{1}]}) \in \Lambda^{(2)}[x] .$$

For any $n \neq 2$, we have \begin{equation*}
(H_{\Pi^D_n}^{\mathcal{G}_{min}})_{[\hat{0},\hat{1}]}=P^{B_n}_{\overline{\mathcal{M}}^D_{0,n}}(x)  
\end{equation*}

In the case $n=2$ however, as remarked above, for the lattice $\Pi^D_2$, the element $\hat{1} \not \in \mathcal{G}_{\min}$ and thus Theorem \cite{Yuzvinsky97} does not hold anymore. More precisely, since $\Pi^D_2 \cong \Pi_2 \times \Pi_2$, the polynomial $(H_{\Pi^D_n}^{\mathcal{G}_{min}})_{[\hat{0},\hat{1}]})$ is the Chow polynomial of the product of the minimal compactifications of $2$ hyperplane arrangements of type $A_1$ and it is thus trivial, i.e. \begin{equation*}
   (H_{\Pi^D_2}^{\mathcal{G}_{min}})_{[\hat{0},\hat{1}]}=\mathds{1}.
\end{equation*}

However, the variety $\overline{\mathcal{M}}^D_{0,n}$ has dimension $1$ and thus, being compact, has non-trivial cohomology. In particular, we can see that $\overline{\mathcal{M}}^D_{0,n} \cong \mathbb{P}^1_{\C}$. 

We consider then $$H^D\coloneqq 1+\sum_{n \geq 1}H^D_n t^n \in \Lambda^{(2)}[x][[t]] .$$ With the same strategy used in type $A$ and $B$ before, we want to find combinatorial expressions for $H^D.$

Notice that
$$( H^{\mathcal{G}_{min}}_{\Pi^D_n} \ast \overline{\chi}_{\Pi^D_n}^{\mathcal{G}_{min}})_{[\hat{0},\hat{1}]}=\sum_{\pi \in \Pi^D_n} \dfrac{2^nn!}{|B_{\pi}|} \Ind_{B_{\pi}}^{B_n}( (H^{\mathcal{G}_{min}}_{\Pi^D_n})_{[\hat{0},\pi]} \otimes (\overline{\chi}_{\Pi^D_n}^{\mathcal{G}_{min}})_{[\pi,\hat{1}]}) .$$

We put $( H^{\mathcal{G}_{min}}_{\Pi^D_n} \ast \overline{\chi}_{\Pi^D_n}^{\mathcal{G}_{min}})_{[\hat{0},\hat{1}],\pi}=\Ind_{B_{\pi}}^{B_n}( (H^{\mathcal{G}_{min}}_{\Pi^D_n})_{[\hat{0},\pi]} \otimes (\overline{\chi}_{\Pi^D_n}^{\mathcal{G}_{min}})_{[\pi,\hat{1}]})$ and we describe now each of these terms.

We treat separately the case $n=2$ and $n>2$.

\textbf{I case $n>2$:}

From the identification of eq. (\ref{intervalstypeD}), we deduce that we have an equality in $\Vrep(B_{\pi})[x]$, \begin{equation}
    (H^{\mathcal{G}_{min}}_{\Pi^D_n})_{[\hat{0},\pi]}=(H^{\mathcal{G}_{min}}_{\Pi^D_{\pi(0)/2}})_{[\hat{0},\hat{1}]} \boxtimes  \prod_{j \in \supp(\lambda)}(H^{\mathcal{G}_{min}}_{\Pi_j})_{[\hat{0},\hat{1}]}^{\boxtimes m_{j,\lambda}}.
\end{equation} 
Here for each $j \in \supp(\lambda)$ we consider $(H^{\mathcal{G}_{min}}_{\Pi_j})_{[\hat{0},\hat{1}]}^{\boxtimes m_{j,\lambda}}$ as a $S_j \wr B_{m_{j,\lambda}}$-representation, where the action of $B_{m_{j,\lambda}}$ factorizes through $S_{m_{j,\lambda}}.$

Under the assumption that $n>2$, we have that $\hat{1} \in \mathcal{G}_{\min}$ and therefore we have $(\overline{\chi}_{\Pi^D_n}^{\mathcal{G}_{min}})_{[\pi,\hat{1}]}=\overline{\chi}_{[\pi,\hat{1}]}$. Through the identification of eq.(\ref{intervalstypeD}), we have thus  
$$(\overline{\chi}_{\Pi^D_n}^{\mathcal{G}_{min}})_{[\pi,\hat{1}]}=\begin{cases}
    \overline{\chi}_{\Pi^B_{l(\lambda)}} \text{ if } |Z|>0 \\
    \overline{\chi}_{\mathcal{L}_{m_{1,\lambda},l(\lambda)}} \text{ if } |Z|=0
\end{cases} .$$

We have these $2$ subcases.

\textbf{Subcase $1$ $\pi(0)>0$:}

We have have an inclusion $\displaystyle\prod_{j \in \supp(\lambda)}B_{m_{\lambda,j}} \subseteq B_{l(\lambda)}$ and thus for each $V$ representation of $B_{l(\lambda)}$ the restriction $\Res^{B_{l(\lambda)}}_{\prod_{j \in \supp(\lambda)}B_{m_{\lambda,j}}}V $ can be endowed with a natural structure of a $B_{\pi}$-representation. We have thus an equality in $\Vrep(B_{\pi})[x]$:
\begin{equation*}
(\overline{\chi}_{\Pi^D_n}^{\mathcal{G}_{min}})_{[\pi,\hat{1}]} \cong  \Res^{B_{l(\lambda)}}_{\prod_{j \in \supp(\lambda)}B_{m_{\lambda,j}}}(\overline{\chi}_{\Pi^B_{l(\sigma)}}^{\mathcal{G}_{min}})_{[\hat{0},\hat{1}]}  
\end{equation*}
Put $\overline{B}_{\pi}=\displaystyle \prod_{j \in \supp(\lambda)} B_{jm_{j,\lambda}}$. We deduce that we have 
\begin{align*}
    &( H^{\mathcal{G}_{min}}_{\Pi^D_n} \ast \overline{\chi}_{\Pi^D_n}^{\mathcal{G}_{min}})_{[\hat{0},\hat{1}],\pi}=
    \\
    &= \begin{multlined}[t] \Ind^{B_n}_{ B_{\pi}}(H^{\mathcal{G}_{min}}_{\Pi^D_{\pi(0)/2}})_{[\hat{0},\hat{1}]} \boxtimes \prod_{j \in \supp(\lambda)}(H^{\mathcal{G}_{min}}_{\Pi_j})_{[\hat{0},\hat{1}]}^{\boxtimes m_{j,\lambda}} \otimes \Res^{B_{l(\lambda)}}_{\prod_{j \in \supp(\lambda)}B_{m_{\lambda,j}}}(\overline{\chi}_{\Pi^B_{l(\sigma)}}^{\mathcal{G}_{min}})_{[\hat{0},\hat{1}]} 
    \end{multlined}
    \\
&= \begin{multlined}[t] \Ind^{B_n}_{B_{\pi(0)/2} \times \overline{B}_{\pi}} \Ind^{B_{\pi(0)/2} \times\overline{B}_{\pi}}_{B_{\pi}} (H^{\mathcal{G}_{min}}_{\Pi^D_{\pi(0)/2}})_{[\hat{0},\hat{1}]} \boxtimes \prod_{j \in \supp(\lambda)}(H^{\mathcal{G}_{min}}_{\Pi_j})_{[\hat{0},\hat{1}]}^{\boxtimes m_{j,\lambda}} \otimes \\
 \otimes\Res^{B_{l(\lambda)}}_{\prod_{j \in \supp(\lambda)}B_{m_{j,\lambda}}}(\overline{\chi}^{\mathcal{G}_{min}}_{\Pi^B_{l(\sigma)}}) 
 \end{multlined}
\end{align*}

Taking characters, we have thus \begin{equation}
\label{typeDZnot0}\ch\left( (H^{\mathcal{G}_{min}}_{\Pi^D_n} \ast \overline{\chi}^{\mathcal{G}_{min}})_{[\hat{0},\hat{1}],\pi}\right)=H^D_{\pi(0)/2}m_{l(\lambda)} \left(\Res^{\Lambda^{(2)}_{l(\lambda)}}_{\Lambda^{(2)}|(\supp(\lambda))|}\overline{\chi}^{B}_{l(\lambda)}[\prod_{j \in \supp(\lambda)}H^A_j[X+Y]]\right)
\end{equation}

\textbf{Subcase $2$ $\pi(0)=0$:}

Notice that, for any $l \in \N$ and $0 \leq a \leq l$, the group $B_a \times B_{l-a}$ acts on $M(\mathcal{D}_{a,l})$, where $B_a$ acts on the first $a$ coordinates and $B_{l-a}$ on the last $l-a$ coordinates. We now describe the action of the group $B_{a} \times B_{l-a}$ on the cohomology $H^{\bullet}(M(\mathcal{D}_{a,l}))$. Notice that there are two cases here. Either $a=0$ and thus $M(\mathcal{D}_{a,l})=M(\mathcal{B}_l)$ or we have an open-closed decomposition $$M(\mathcal{D}_{a,l})=M(\mathcal{B}_l) \bigsqcup \underbrace{M(\mathcal{B}_{l-1}) \sqcup \cdots \sqcup M(\mathcal{B}_{l-1})}_{a \text{ copies }} $$ where the $B_a$-factor acts by permutation of the $a$-copies of $M(\mathcal{B}_{l-1})$. We have the corresponding long-exact sequence in compactly-supported cohomology:
\begin{align*}
&\cdots  \rightarrow H_c^{i-1}(\underbrace{M(\mathcal{B}_{l-1}) \sqcup \cdots \sqcup M(\mathcal{B}_{l-1})}_{a \text{ copies }})  \rightarrow H_c^i(M(\mathcal{B}_l)) \rightarrow \\ &H_c^i(M(\mathcal{D}_{a,l})) \rightarrow H_c^i(\underbrace{M(\mathcal{B}_{l-1}) \sqcup \cdots \sqcup M(\mathcal{B}_{l-1})}_{a \text{ copies }}) \rightarrow  
\end{align*}

From the properties of this kind of long-exact sequences for hyperplane arrangements (see for instance), we deduce that  all the maps $H_c^i(M(\mathcal{D}_{a,l})) \rightarrow H_c^i(\underbrace{M(\mathcal{B}_{l-1})\sqcup \cdots \sqcup M(\mathcal{B}_{l-1})}_{a \text{ copies }})$ are all $0$ and we find thus that \begin{equation*}
      H_c^i(M(\mathcal{B}_{l})) \cong H_c^i(M(\mathcal{D}_{a,l})) \oplus H_c^{i-1}(\underbrace{M(\mathcal{B}_{l-1}) \sqcup \cdots \sqcup M(\mathcal{B}_{l-1})}_{a \text{ copies }})
  \end{equation*}
   
Notice that, as a $B_a \times B_{l-a}$-representation, we have that \begin{equation}
  \label{cohomologyLan}  H_c^{i-1}(\underbrace{M(\mathcal{B}_{l-1}) \sqcup \cdots \sqcup M(\mathcal{B}_{l-1})}_{a \text{ copies }}) \cong \Ind^{B_a \times B_{l-a}}_{B_1 \times B_{a-1} \times B_{l-a}}(1 \boxtimes H_c^{i-1}(M(\mathcal{B}_{l-1})))
\end{equation}

From eq.(\ref{cohomologyLan}), we deduce thus that we have an equality in $\Vrep(B_a \times B_{l-a})[x]$

\begin{equation}
\label{latticeLal}
\overline{\chi}^{\mathcal{L}_{a,l}}=\Res^{B_l}_{B_a \times B_l}(\overline{\chi}^{\Pi^B_l})+\Ind^{B_a \times B_{l-a}}_{B_1 \times B_{a-1} \times B_{l-a}}(1 \boxtimes \Res^{B_{l-1}}_{B_{a-1} \times B_{l-a}}\overline{\chi}^{\Pi^B_{l-1}})
\end{equation}

We have thus that, if $m_{1,\lambda}=0$, \begin{equation*}
(\overline{\chi}_{\Pi^D_n}^{\mathcal{G}_{min}})_{[\pi,\hat{1}]}=\Res^{B_{l(\lambda)}}_{\prod_{j \in \supp(\lambda)}B_{m_{\lambda,j}}}(\overline{\chi}_{\Pi^B_{l(\sigma)}})    
\end{equation*}
and if $m_{1,\lambda}>0$, 

\begin{equation*}
(\overline{\chi}_{\Pi^D_n}^{\mathcal{G}_{min}})_{[\pi,\hat{1}]}=\Res^{B_{l(\lambda)}}_{\prod_{j \in \supp(\lambda)}B_{m_{\lambda,j}}}(\overline{\chi}_{\Pi^B_{l(\sigma)}})+ \Ind_{B_1 \times B_{m_{1,\lambda}-1} \times \prod_{j \in \supp(\lambda) \setminus \{1\}}B_{m_{j,\lambda}}}^{\prod_{j \in \supp(\lambda)}B_{m_{\lambda,j}}}(\overline{\chi}^{\Pi^B_{l(\lambda)-1}})        
\end{equation*}

Let $\Tilde{\lambda}$ be the partition of $n-1$ such that $m_{j,\lambda}=m_{j,\widetilde{\lambda}}$ for all $j \geq 2$ and $m_{1,\lambda}=m_{1,\widetilde{\lambda}}+1$. Notice that $l(\widetilde{\lambda})=l(\lambda)-1$ and \begin{equation*}p_1[X_1]\Res^{B_{l(\widetilde{\lambda})}}_{\prod_{j \in \supp(\lambda)}B_{m_{\widetilde{\lambda},j}}}(\overline{\chi}^{\Pi^B_{l(\widetilde{\lambda})}})=\Ind_{B_1 \times B_{m_{1,\lambda}-1} \times \prod_{j \in \supp(\lambda) \setminus \{1\}}B_{m_{j,\lambda}}}^{\prod_{j \in \supp(\lambda)}B_{m_{\lambda,j}}}(\overline{\chi}^{\Pi^B_{l(\lambda)-1}}) \end{equation*}   

With the same line of reasoning used above, we see that we have 
\begin{align}
    \ch\left( (H^{\mathcal{G}_{min}}_{\Pi^D_n} \ast \overline{\chi}_{\Pi^D_n}^{\mathcal{G}_{min}})_{[\hat{0},\hat{1}],\pi}\right)=m_{l(\lambda)} \left(\Res^{\Lambda^{(2)}_{l(\lambda)}}_{\Lambda^{(2)}|(\supp(\lambda))|}\overline{\chi}^{B}_{l(\lambda)}[\prod_{j \in \supp(\lambda)}H^A_j[X+Y]]\right)+ \\ 
     p_1[X][H^A_1[X+Y]]m_{l(\widetilde{\lambda})} \left(\Res^{\Lambda^{(2)}_{l(\widetilde{\lambda})}}_{\Lambda^{(2)}|(\supp(\widetilde{\lambda}))|}\overline{\chi}^{B}_{l(\widetilde{\lambda})}[\prod_{j \in \supp(\widetilde{\lambda})}H^A_j[X+Y]]\right).
     \label{typeDcaseZ=0}
\end{align}

Notice that $H^A_1[X+Y]=p_1[X+Y]$ and $p_1[X][p_1[X+Y]]=p_1[X]$. Notice that, for any $m$, we have a bijection $$\{\lambda \in \mathcal{P}_m \ | \ m_{1,\lambda}>0\} \to \{\mu \in \mathcal{P}_{m-1}\}$$ $$\lambda \to \widetilde{\lambda} .$$

\textbf{II case $n=2$:}

In the case of $n=2$, we have that $|\Pi^D_2|=4$ and $\Pi^D_2=\{\hat{0},\pi_1,\pi_2,\hat{1}\}$, where $\pi_1=\{\{1,2\},\{\overline{1},\overline{2}\}\}$ and $\pi_2=\{\{1,\overline{2}\},\{\overline{1},2\}\}$. Notice that in this case   $\Pi^D_2\cong [\hat{0},\pi_1] \times [\hat{0},\pi_2] \cong \Pi_2 \times \Pi_2$. In particular, $\hat{1} \not \in \mathcal{G}_{\min}$ and $\mathcal{G}_{\min}=\{\pi_1,\pi_2\}$.  Recall that $B_2=\Z_2 \wr S_2$. We have thus an equality in  $\Vrep(B_2)[x]$ \begin{equation*}
    \overline{\chi}_{\Pi^D_2}^{\mathcal{G}_{\min}}=(-1)\overline{\chi}_{\Pi_2} \otimes \overline{\chi}_{\Pi_2} \otimes \sgn_{\{\pi_1,\pi_2\}} 
\end{equation*}

Since $\overline{\chi}_{\Pi_2}$ is the trivial $S_2$ representation, we deduce that $\overline{\chi}_{\Pi_2} \otimes \overline{\chi}_{\Pi_2}$ is the trivial $B_2$-representation. We also have $\sgn_{\{\pi_1,\pi_2\}}=V_{(\emptyset,(2))}$, from which we deduce
\begin{equation*}
  \overline{\chi}_{\Pi^D_2}^{\mathcal{G}_{\min}}=-V_{(\emptyset,(2))}  
\end{equation*}

Since $\overline{\chi}_{\Pi_2}=\mathrm{1}$, from the description of the irreducible characters of $B_n$, we deduce that we have \begin{equation*}
    \ch(\overline{\chi}_{\Pi^D_2}^{\mathcal{G}_{\min}})=-s_{(\emptyset,2)}
\end{equation*}

Putting $a=l=2$ in eq.(\ref{latticeLal}), we deduce that 

\begin{equation*}
    \ch(\overline{\chi}_{\Pi^D_2})=\ch(\overline{\chi}_{\Pi^B_2})+p_1[X]\ch(\overline{\chi}_{\Pi^B_1})
\end{equation*}

Notice that $\ch(\overline{\chi}_{\Pi^B_1})=p_1[X]$ and thus \begin{equation*}
\ch(\overline{\chi}_{\Pi^D_2})=\ch(\overline{\chi}_{\Pi^B_2})+p_1[X]^2.    
\end{equation*}

A direct computation shows that 
\begin{equation*}
 \overline{\chi}_{\Pi^B_2}=x\mathds{1}-\mathds{1}-V_{(1^2,\emptyset)}-V_{(\emptyset,2)}.  
\end{equation*}

 From eq. and eq. we deduce that we have \begin{equation}
 \label{casen2D}
    \ch(\overline{\chi}_{\Pi^B_2})+p_1[X]^2+s_{(\emptyset,2)}=xs_{(2,\emptyset)}[X]
\end{equation}

Our main result concerning the generating series $H^D$ is the following:
\begin{theorem}
    We have the following identity
    \begin{equation}
    \label{main-equation-typeD}
        H^D \overline{\chi}^B[H^A[X+Y]]=-1+xs_{(2,\emptyset)}t^2
    \end{equation}
\label{main-result-typeD}
\end{theorem}

\begin{proof}
From  eq.(\ref{typeDcaseZ=0}) and eq.(\ref{typeDZnot0}), we deduce that, for each $n \neq 2$, we deduce that \begin{align*}
    &\ch(( H^{\mathcal{G}_{min}}_{\Pi^B_n} \ast \overline{\chi}_{\Pi^B_n}^{\mathcal{G}_{min}})_{[\hat{0},\hat{1}]})=\sum_{\substack{(s,\lambda) \in \N \times \mathcal{P} \\ s+|\lambda|=n \\ s \neq \{0,1\}}}H^D_s m_{l(\lambda)} \left(\Res^{\Lambda^{(2)}_{l(\lambda)}}_{\Lambda^{(2)}|(\supp(\lambda))|}\overline{\chi}^{B}_{l(\lambda)}[\prod_{j \in \supp(\lambda)}H^A_j[X+Y]]\right)+\\
    &+\sum_{\lambda \in \mathcal{P}_n}m_{l(\lambda)} \left(\Res^{\Lambda^{(2)}_{l(\lambda)}}_{\Lambda^{(2)}|(\supp(\lambda))|}\overline{\chi}^{B}_{l(\lambda)}[\prod_{j \in \supp(\lambda)}H^A_j[X+Y]]\right)+\\ &\sum_{\substack{\lambda \in \mathcal{P}_n \\ m_{1,\lambda}>0}}p_1[X]m_{l(\widetilde{\lambda})} \left(\Res^{\Lambda^{(2)}_{l(\widetilde{\lambda})}}_{\Lambda^{(2)}|(\supp(\widetilde{\lambda}))|}\overline{\chi}^{B}_{l(\widetilde{\lambda})}[\prod_{j \in \supp(\widetilde{\lambda})}H^A_j[X+Y]]\right)
\end{align*}
Notice that \begin{align*}
&\sum_{\substack{\lambda \in \mathcal{P}_n \\ m_{1,\lambda}>0}}p_1[X]m_{l(\widetilde{\lambda})} \left(\Res^{\Lambda^{(2)}_{l(\widetilde{\lambda})}}_{\Lambda^{(2)}|(\supp(\widetilde{\lambda}))|}\overline{\chi}^{B}_{l(\widetilde{\lambda})}[\prod_{j \in \supp(\widetilde{\lambda})}H^A_j[X+Y]]\right)= \\ &\sum_{\mu \in \mathcal{P}_{n-1}} m_{l(\mu)} \left(\Res^{\Lambda^{(2)}_{l(\mu)}}_{\Lambda^{(2)}|(\supp(\mu))|}\overline{\chi}^{B}_{l(\mu)}[\prod_{j \in \supp(\mu)}H^A_j[X+Y]]\right)   
\end{align*}

Since $H^D_0=1$ and $H^D_1=p_1[X]$, we deduce that, if $n \neq 2$, \begin{equation}
\label{prooftypeD1}
\ch(( H^{\mathcal{G}_{min}}_{\Pi^B_n} \ast \overline{\chi}_{\Pi^B_n}^{\mathcal{G}_{min}})_{[\hat{0},\hat{1}]})=\sum_{\substack{(s,\lambda) \in \N \times \mathcal{P} \\ s+|\lambda|=n}}H^D_sm_{l(\lambda)} \left(\Res^{\Lambda^{(2)}_{l(\lambda)}}_{\Lambda^{(2)}|(\supp(\lambda))|}\overline{\chi}^{B}_{l(\lambda)}[\prod_{j \in \supp(\lambda)}H^A_j[X+Y]]\right)     
\end{equation}
Similarly, we have that, if $n=2$, \begin{align}
&\ch(( H^{\mathcal{G}_{min}}_{\Pi^B_2} \ast \overline{\chi}_{\Pi^B_2}^{\mathcal{G}_{min}})_{[\hat{0},\hat{1}]})= \nonumber \\ & \sum_{\substack{(s,\lambda) \in \N \times \mathcal{P} \\ s+|\lambda|=2}}H^D_sm_{l(\lambda)} \left(\Res^{\Lambda^{(2)}_{l(\lambda)}}_{\Lambda^{(2)}|(\supp(\lambda))|}\overline{\chi}^{B}_{l(\lambda)}[\prod_{j \in \supp(\lambda)}H^A_j[X+Y]]\right)-xs_{(2,\emptyset)} \label{prooftypeD}  
\end{align}

From Theorem \ref{compositional-formulatypeB}, eq.(\ref{prooftypeD1}), eq.(\ref{prooftypeD}) we deduce eq.(\ref{main-equation-typeD}).
\end{proof}

\section{Cohomology of Arrangements of Type \texorpdfstring{$B_n$}{Bn}}
\label{sec:CohomologyTypeB}

Recall that, every Coxeter group $W$ is naturally associated with an arrangement of hyperplanes $\mathcal{A}(W)$.
The topology of the complement $M(W)$ and of its natural $W$-action have always been the focus of considerable attention in the literature, starting from the seminal work of Brieskorn \cite{brieskorn}, who described the cohomology of $M(W)$.
Among the subsequent works on this topic, it is worth mentioning the article of Lehrer and Solomon \cite{lehrer-solomon} who described the $W$-action on cohomology in type $A$. Later, Lehrer studied the case of type $B$ \cite{lehrer1} and Douglass \cite{DouglassTypeBn} gave a more refined decomposition of the cohomology in this case. 

We recall the main result of \cite{DouglassTypeBn} on the action of the group $B_n$ on the cohomology $H^p (M(\mathcal{B}_n)) \cong \OS_{\mathcal{B}_n,p}$. Let $\pi = (1^{a_1}, 2^{a_2}, \dots) ,\lambda = (1^{b_1}, 2^{b_2}, \dots)$ be two partitions and define the groups
\[ Y_n \coloneqq \mathbb{Z}_2 \times C_n \subseteq B_n \]
where $C_n$ is a positive $n$-cycle and $\mathbb{Z}_2$ is the center.
Define the group $Y_{(\pi, \lambda)}$ as the product
\[ Y_{\pi, \lambda} \coloneqq \bigtimes_i Y_i \wr S_{a_i} \times \bigtimes_j Y_j \wr S_{b_j}, \]
and its character
\[ \zeta_{\pi, \lambda} \coloneqq \left( \boxtimes_i (1 \boxtimes \zeta_i ) \boxtimes \sgn_{a_i} \right) \boxtimes \left( \boxtimes_j (1 \boxtimes \zeta_j ) \boxtimes  \mathds{1}_{b_j} \right).
\]
\begin{theorem}[\cite{DouglassTypeBn}*{Corollary 3.9}]
    For all $p \leq n$ we have
    \[ H^p (M(\mathcal{B}_n)) \cong \OS_{\mathcal{B}_n,p} \cong V_{(1^n),\emptyset} \otimes \bigoplus_{\substack{\lvert \pi\rvert + \lvert \lambda \rvert = n \\ \ell(\pi)=n-p}} \Ind_{Y_{\pi, \lambda}}^{B_n} \zeta_{\pi, \lambda} \]
\end{theorem}
 In what follows, we put $\chi^B\coloneqq \displaystyle \sum_{n,p} (-1)^{p} \ch (H^p(M(\mathcal{B}_n))) x^{n-p}t^n $. For every $n$, let $\Lie_n$ be the degree $n$-part of the free Lie algebra in $n$ generators. Notice that $S_n$ acts on $\Lie_n$ by permutation of the generators. We can then consider $\Lie_n$ as an $S_n$-representation. We put $\Lie_{\epsilon} = \sum_{i=1}^\infty \ch (\Lie_i) \epsilon^it^i .$

From the previous result we obtain:
\begin{theorem}
\label{formula-chiB}
The equivariant generating series of characteristic polynomials of type $B$ is
    \begin{align*}
    \chi^B &= \omega \left( \Exp \left( (1- x) \bigl( \Lie_{\epsilon} [X+Y])^\even \right) \right)\Big|_{\epsilon=-1} \\
    &=  \prod_{n \geq 1} \left( \frac{(1+ t^np_n[X+Y])(1+ t^n p_n[X- Y])}{1+ t^{2n} p_{2n}[X+Y]} \right)^{T_n^B(x)},
    \end{align*}
    where
    \[ T_n^B(x)= \frac{1}{2n}\sum_{\substack{d \mid n \\ d \textnormal{ odd}}} \mu (d) (x^{\frac{n}{d}}-1).\]
\end{theorem}
\begin{proof}
By the previous theorem we have
    \begin{align*}
         \omega(\chi^B) &=  \sum_{\pi, \lambda} (-1)^{\lvert \pi \rvert +\lvert \lambda \rvert - \ell(\pi)} \ch(\Ind_{Y_{\pi, \lambda}}^{B_{\lvert \pi \rvert +\lvert \lambda \rvert}} \zeta_{\pi, \lambda}) x^{\ell(\pi)}t^{|\lambda|+|\pi|}  \\
        &= \left( \sum_\pi \prod_i (- x)^{a_i} e_{a_i}[t^i \epsilon^i\ch(\Ind_{Y_i}^{B_i} \mathds{1} \boxtimes \zeta_i) ]\right)
        \left( \sum_\lambda \prod_i h_{b_i}[\epsilon ^it^i\ch(\Ind_{Y_i}^{B_i} \mathds{1} \boxtimes \zeta_i) ]\right)\Big|_{\epsilon=-1},
    \end{align*}
    where we used \Cref{thm:Ind_repr_wr}. The induced representation $ \Ind_{C_i}^{S_i} \zeta_i$ is the representation $\Lie_i$ and by \Cref{lem:ind_Z_2xS_n}:
    \begin{align*}
        \omega(\chi^B) &= \left( \sum_\pi \prod_i (-x)^{a_i}e_{a_i}[(\Lie_i[X+Y])^\even t^i\epsilon ^i]\right)
        \left( \sum_\lambda \prod_i h_{b_i}[(\Lie_i[X+Y])^\even t^i \epsilon^i]\right)\Big|_{\epsilon=-1}\\
        &= \left(  \prod_{i=1}^{\infty} \sum_{a=0}^\infty (-x)^{a} e_a [(\Lie_i[X+Y])^\even t^i \epsilon^i] \right)
        \left( \prod_{i=1}^{\infty} \sum_{a=0}^\infty h_{a}[(\Lie_i[X+Y])^\even t^i \epsilon^i]\right)\Big|_{\epsilon=-1} \\
        &= \prod_{i=1}^{\infty} \Exp \left( -x \epsilon^{i}t^i (\Lie_i[X+Y])^\even \right)
         \prod_{i=1}^{\infty} \Exp \left( \epsilon^{i}t^i (\Lie_i[X+Y])^\even  \right)\Big|_{\epsilon=-1} \\
        &= \Exp \left( (1- x) \bigl( \Lie_{\epsilon} [X+Y])^\even \right).
    \end{align*}

    By \cite{GetzlerOperads}*{Proposition 5.3} we have $\Lie_\epsilon = - \sum_{n \geq 1} \frac{\mu(n)}{n}\log(1- \epsilon^n p_n)$ and 
    \begin{align*}
        (\log(1-\epsilon^nt^np_n)[X+Y])^\even = \begin{cases}
            \log(1-\epsilon^nt^np_n)[X+Y] & \textnormal{if } n \textnormal{ even,} \\
            \frac{1}{2}(\log(1-\epsilon^nt^np_n)[X+Y]+\log(1-\epsilon^nt^np_n)[X-Y]) & \textnormal{if } n \textnormal{ odd.}  
        \end{cases}
    \end{align*}
    We have three factors, the first one for $n$ even is
    \begin{align*}
        &\Exp  \left[ (x-1) \sum_{n \textnormal{ even}}  \frac{\mu(n)}{n}  \log(1- \epsilon^n t^n p_n)[X+Y] \right] \\
        &=\exp \left(\sum_{k\geq 1} \frac{p_k}{k} \right) \left[ -\frac{x-1}{2} \sum_{n \textnormal{ odd}}  \frac{\mu(n)}{n}  \log(1- \epsilon^{2n}t^{2n} p_{2n})[X+Y] \right] \\
        &= \exp \left( -\sum_{k\geq 1} \frac{x^k-1}{2k} \sum_{n \textnormal{ odd}}  \frac{\mu(n)}{n} \Bigl( \log(1- \epsilon^{2nk}t^{2nk} p_{2nk})[X+Y] \Bigr) \right) \\
        &= \exp \left( - \sum_{m\geq 1} \sum_{\substack{d \mid m \\ d \textnormal{ odd}}}  \frac{x^{\frac{m}{d}}-1}{2m}  \mu (d)  \log(1- \epsilon^{2m}t^{2m} p_{2m})[X+Y]  \right) \end{align*}
        Thus, evaluating at $\epsilon=-1$, we obtain
        \begin{align*}
         &= \prod_{m \geq 1} \left(1- (-1)^{2m}t^{2m} p_{2m}[X+Y] \right)^{- T_m^B(x)}.
    \end{align*}
    The second one is
    \begin{align*}
        &\exp \left(\sum_{k\geq 1} \frac{p_k}{k} \right) \left[ (x-1) \sum_{n \textnormal{ odd}}  \frac{\mu(n)}{2n}  \log(1- \epsilon^n t^n p_n)[X+Y] \right] \\
        &= \exp \left( \sum_{k\geq 1} \frac{x^k-1}{k} \sum_{n \textnormal{ odd}}  \frac{\mu(n)}{2n} \Bigl( \log(1- \epsilon^{nk} t^{nk} p_{nk})[X+Y] \Bigr) \right) \\
        &= \exp \left( \sum_{m\geq 1} \sum_{\substack{d \mid m \\ d \textnormal{ odd}}}  \frac{x^{\frac{m}{d}}-1}{2m}  \mu (d)  \log(1- \epsilon^{m}t^m p_{m})[X+Y]  \right) \end{align*}
        Thus, evaluating at $\epsilon=-1$, we obtain
        \begin{align*}
         &= \prod_{m \geq 1} \left(1- (-1)^{m} t^m p_{m}[X+Y] \right)^{ T_m^B(x)}.
    \end{align*}
    Similarly the last one is $\prod_{m \geq 1} \left(1- (-1)^{m}t^m p_{m}[X-Y] \right)^{ T_m^B(x)}$.
    We have obtained 
    \begin{align*}
        \omega(\chi^B) &=  \prod_{m \geq 1} \left( \frac{(1- (-1)^{m}t^m p_{m}[X+Y])(1- (-1)^{m}t^m p_{m}[X- Y])}{(1- (-1)^{2m}t^{2m} p_{2m}[X+Y])} \right)^{T_m^B(x)}.
    \end{align*}
    Applying $\omega$ on both sides we obtained the claimed identity.
\end{proof}

\begin{remark}
Notice that the same reasoning used in the proof above shows that \begin{align}
    &\prod_{n \geq 1} \left( \frac{(1+ t^np_n[X+Y])(1+ t^n p_n[X- Y])}{1+ t^{2n} p_{2n}[X+Y]} \right)^{T_n^B(x)}=\Exp((x-1)\Logg((1+p_1[X+Y]t))_{\even})=\nonumber \\ 
    &(1+p_1[X+Y]t)^{\langle x-1 \rangle_B}
    \label{eq-exp-typeB}
\end{align}

From eq.(\ref{eq-exp-typeB}) and \Cref{formula-chiB}, we deduce the following formula for $\overline{\chi}^B(t)$

\begin{equation}
\label{eq-characteristic-reduced-typeB}
    \overline{\chi}^B=\dfrac{(1+p_1[X+Y]t)^{\langle x-1 \rangle_B}-x}{x-1}
\end{equation}

\end{remark}

 \subsection{Alternative proof of Theorem \ref{formula-chiB}}

In this paragraph, we sketch an alternative proof of Theorem \ref{formula-chiB}  by counting points over finite fields of hyperplane arrangements of type $B$. This technique  has been applied in several papers to the calculation of the cohomology of the complements of hyperplane arrangements; see, for example, \cite{ATHANASIADIS1996193},\cite{Lehrer} and \cite{Kisin-Lehrer} for an extension to the equivariant case. Specifically, in the latter paper, the authors use the counting of points over finite fields to provide an alternative proof of Getzler’s formula of eq.(\ref{equation_chitypeA}) for $\overline{\chi}^A$, mentioned earlier. We apply the same approach in type $B$.

Recall that we have $$\chi_n^B(x)=1/x^n \ch(P_{M(\mathcal{B}_n),c}^{B_n})(-x) .$$

The latter polynomial can be rewritten as follows. For any $w \in B_n$, let $$m_w(x)=\sum_i (-x)^i \tr(w | H_c^i(M(\mathcal{B}_n),\Q)) $$ and for any $(\lambda,\mu) \in \mathcal{P}^{(2)}_n$, put $m_{(\lambda,\mu)}(x)=m_{w}(x)$ for any element $w$ in the conjugacy class indexed by $(\lambda,\mu)$.  We have thus $$\chi^B_n(x)= \sum_{(\lambda,\mu) \in \mathcal{P}_n^{(2)}}\dfrac{m_{(\lambda,\mu)}(x)}{x^n}z_{(\lambda,\mu)}^{-1}p_{(\lambda,\mu)} .$$

It is thus enough to understand the polynomials $m_{(\lambda,\mu)}(x)$. These polynomials can be understood using finite fields in the following way. Let $q=p^m$ be the power of an odd prime $p$ and  $M(\mathcal{B}_{n,q})=\{(x_1,\dots,x_n) \in (\overline{\F}^{\times}_q)^n \ | \ x_i \neq \pm x_j\}$. Notice that $B_n$ acts on $M(\mathcal{B}_{n,q}).$ Let $F:M(\mathcal{B}_{n,q}) \to M(\mathcal{B}_{n,q})$ the Frobenius map, i.e. $F(x_1,\dots,x_n)=(x_1^q,\dots,x_n^q)$ and we put $wF=w \circ F$. In this setting, we have 
\begin{equation}
\label{finite-fields-equation}
    m_w(q)q^{-n}=\#M(\mathcal{B}_{n,q})^{wF},
\end{equation}
i.e. it is enough to understand the quantities $\#(M(\mathcal{B}_{n,q}))^{wF}$ for any $w$ and every $q$. Formula (\ref{finite-fields-equation}) is a consequence, for instance of \cite{Lehrer}*{Proposition 2.4} and the results of the second author with Letellier \cite{Letellier-Scognamiglio}*{Theorem 2.7.1}.

We recall some generalities about the Frobenius action here useful to compute these quantities. In what follows, for a Frobenius orbit $O \subseteq \overline{\F}_q^{\times}$ it is a meant a subset $O=\{x^{q^d}\}_{d \in \N_{>0}}$ for an element $x \in \overline{\F}_q^{\times}$. Notice that $|O|=r$ if and only if $x \in \F_{q^r}$ and $x \not \in \F_{q^i}$ for any $i \leq r$. Moreover, two orbits either coincide or are disjoint. We deduce that the number of Frobenius orbits of size $r$, is given by $2rR^B_r(q)$, where $$R^B_r(x)=\dfrac{1}{2r}\sum_{d |r}\mu(d)(x^{\frac{r}{d}}-1) .$$

Notice that for any Frobenius orbit $O$, the set $-O=\{-y\}_{y \in O}$ is a Frobenius orbit too and thus either $O=-O$ or $O $ and $-O$ are disjoint. Moreover, if there is an orbit $O$ of length $r$ such that $O=-O$, we have that $r$ is even and $x^{q^{\frac{r}{2}}}=-x$. 

In this case, it is not complicated to see that the number of elements having Frobenius orbits of length $r$ with $O=-O$ is given by $rT^B_{r/2}(q)$, where  In what follows, put $E^B_r(x)\coloneqq R^B_r(x)-1/2 T^B_{r/2}(x)$ if $r$ is even.

Fix $(\lambda,\mu) \in \mathcal{P}^{(2)}_n$ and $w$ in the corresponding conjugacy class. Suppose first that $\mu=\emptyset$ and $\lambda=(n)$. We have that $x=(x_1,\dots,x_n) \in M(\mathcal{B}_{n,q})^{wF}$ if and only if $x_1^{q^n}=x$ and $x^{q^d} \neq \pm x$ for all $1 \leq d \leq n-1$. We are thus asking for the number of elements having Frobenius orbits $ O \subseteq \overline{\F}_q^{\times}$ of length $n$ such that $O \neq -O$. We have thus that in this case if $n$ is odd $m_{((n),\emptyset)}(q)=2nR^B_n(q)$ and if $n$ is even $m_{((n),\emptyset)}(q)=2nE^B_n(q)$.

More generally, we see that if $\mu=\emptyset$ and $\lambda=(j^{m_j})$, we have that \begin{align}&m_{(\lambda,\emptyset)}(x)=\prod_{j \text{odd }}(2jT^B_j(x))(2jT^B_j(x)-2j) \cdots (2jT^B_j(x)-2j(m_j-1)) \nonumber \\& \prod_{j \text{ even}} (2jE^B_j(x))(2jE^B_j(x)-2j) \cdots (2jE^B_j(x)-2j(m_j-1))= \nonumber\\ &\prod_{j \text{odd }}(2j)^{m_j}j!\binom{T^B_j(x)}{m_j}\prod_{j \text{ even}}(2j)^{m_j}j!\binom{E^B_j(x)}{m_j}=2^{l(\lambda)}z_{\lambda}\prod_{j \text{odd }}\binom{T^B_j(x)}{m_j}\prod_{j \text{ even}}\binom{E^B_j(x)}{m_j} .\label{eq-typeB-fin-fields1}
\end{align}

Suppose now that $\lambda=\emptyset$ and $\mu=(n)$. We have that $x \in M(\mathcal{B}_{n,q})^{wF}$ if and only if $x^{q^n}=-x$ and $x^{q^d} \neq \pm x$ for all $1 \leq d \leq n-1$. We are thus asking for the number of elements having Frobenius orbits $O$ of size $2n$ and such that $O=-O$. We thus see that $m_{(\emptyset,(n))}=2nT^B_{n}(q)$. More generally, as above if $\mu=(j^{m_j})$, \begin{equation}m_{(\emptyset,\mu)}(x)=2^{l(\mu)}z_{\mu}\prod_{j}\binom{T^B_{j}(x)}{m_j} .\label{eq-typeB-fin-fields2}\end{equation} From the description above, we also deduce that $m_{(\lambda,\mu)}(x)=m_{(\lambda,\emptyset)}(x)m_{(\emptyset,\mu)}(x)$. From  eq.(\ref{eq-typeB-fin-fields1}) and eq.(\ref{eq-typeB-fin-fields2}), we deduce that \begin{align}
   & 1+\sum_{n \geq 1}t^n \sum_{(\lambda,\mu) \in \mathcal{P}^{(2)}}m_{(\lambda,\mu)}(x)z_{(\lambda,\mu)}^{-1}p_{\lambda,\mu}=\nonumber\\ & 1+\sum_{(\lambda,\mu) \in \mathcal{P}^{(2)}}\prod_{j \text{ odd }}\binom{T^B_j(x)}{m_{j,\lambda}}\binom{T^B_{j}(x)}{m_{j,\mu}}\prod_{j \text{ even}}\binom{E^B_j(x)}{m_{j,\lambda}}\binom{T^B_{j}(x)}{m_{j,\mu}}t^{|\lambda|+|\mu|}=\nonumber \\
& \prod_{n \text{ odd}}(1+t^np_n[X+Y])^{T^B_n(x)}(1+t^np_n[X-Y])^{T^B_{n}(x)}\label{eq-arith-proof0} = \\ 
& \prod_{n \text{ even}}(1+t^np_n[X+Y])^{R^B_n(x)-1/2T^B_{n/2}(x)}(1+t^n p_n[X-Y])^{T^B_n(x)} \label{eq-arith-proof1}
\end{align}

Notice that, for any $d$ such that $4 |d$, we have $\mu(d)=0$. We deduce that for any $n$ even, we have $R^B_n(x)=T^B_n(x)-1/2T^B_{n/2}(x)$, from which we see that the product in eq.(\ref{eq-arith-proof0}) and \ref{eq-arith-proof1} above is equal to 
\begin{equation*}
    \prod_{n \geq 1}\left( \frac{(1+t^n p_n[X+Y])(1+t^n p_n[X- Y])}{1+t^{2n}p_{2n}[X+Y]} \right)^{T_n^B(x)}.
\end{equation*}

\subsection{Combinatorial formulas for \texorpdfstring{$H^{B}(t)$}{HB(t)} and \texorpdfstring{$H^D(t)$}{HD(t)}}

\subsubsection{Type \texorpdfstring{$B$}{B}}
From Remark \ref{inverted-formula-A} and Lemma \ref{plethysmsn-bn}, we deduce that we have \begin{equation}
    -\overline{\chi}^A[X+Y](H^A[X+Y])=p_{((1),\emptyset)}t
\end{equation}

Notice that, for any $H \in \Lambda^{(2)}[[t]]$, we have $H[p_{((1),\emptyset)}t]=H$. From Theorem \ref{identity-typeB} and eq.(\ref{eq-characteristic-reduced-typeB}), we deduce thus the following result.
\begin{theorem}
\label{thm:inverted-equation-B}
There is an equality
\begin{equation}
\label{inverted-equation-B}
H^B=\dfrac{-1}{\overline{\chi}^B(H^A[X+Y])}=\dfrac{1-x}{(1+H_A[X+Y])^{\langle x-1 \rangle_B}-x}.
\end{equation}
\end{theorem}

\begin{remark}
Similarly to what observed in Remark \ref{inverted-formula-A}, Formula \eqref{inverted-equation-B} can also be obtained by swapping the role of $H^B$ and $\overline{\chi}^B$ in the proof of Theorem \ref{identity-typeB}, i.e.\ using that, for every $n \geq 1$, we have  $$( H^{\mathcal{G}_{min}}_{\Pi^B_n} \ast \overline{\chi}_{\Pi^B_n}^{\mathcal{G}_{min}})_{[\hat{0},\hat{1}]}=0.$$
\end{remark}

Given the results of Theorem \ref{compositional-formulatypeB} describing plethystic composition of series with coefficients in $\Lambda^{(2)}$ and the combinatorial descriptions of the plethystic inverse, mentioned in \cref{paragraph-equiv-bell}, we obtain a priori a way to describe combinatorially the coefficients of $H^B(t)$.

Theorem \ref{thm:inverted-equation-B} was already proven in \cite{Henderson}*{Theorem 2.4}, in which the author studies more generally  the arrangements associated with wreath products of cyclic groups. In this latter article, however, the author does not use incidence algebras; instead, his proof relies on species and an explicit stratification of minimal wonderful models. As remarked in the introduction, even the plethystic operations introduced in type $B$ are defined differently, albeit equivalently, and there does not appear to be an explicit mention of their interpretation in terms of representation theory.  

From eq.(\ref{inverted-equation-B}), the same strategy of proof of \cite{PagariaPezzoli}*{Theorem 4.9} shows that we have \begin{theorem}
    The series $H^B$ respects the following differential equation
    \begin{equation*}
        2(1+xp_1[X+Y]-xH^A[X+Y])\dfrac{\partial H^B}{\partial (p_1[X+Y]t)}=H^B(1-x+xH^B)
    \end{equation*}
\end{theorem}

\subsubsection{Type \texorpdfstring{$D$}{D}}

Notice that from eq.(\ref{main-equation-typeD}) and eq.(\ref{inverted-equation-B}) we deduce the following.
\begin{theorem}
For any $n$, we have the following identity in $\Vrep(B_n)[x]$
\begin{equation}\label{inducedDn}
    H^D_n=H^B_n-x\Ind_{B_2 \times B_{n-2}}^{B_n}(\mathds{1} \boxtimes H^B_{n-2}).
\end{equation}
\end{theorem}
Eq. \eqref{inducedDn} above is an equivariant generalization of the same identity for the corresponding Poincaré polynomials, already shown by Gaiffi \cite{Gaiffi-Serventi}, who used the description of the Gr{\"o}bner basis  of $H^{\mathcal{G}_{\min}}_{\Pi^D_n}$ in terms of building sets and the inclusion $\Pi^D_n \subseteq \Pi^B_n$. A similar strategy can be used to improve Gaiffi's description equivariantly and obtain eq. \eqref{inducedDn} directly, from which eq. (\ref{main-equation-typeD}) can be obtained using Theorem \ref{thm:inverted-equation-B}.

Moreover, using the same strategy of the proof of Theorem \ref{thm:inverted-equation-B}, we obtain the following compositional formula for the series $H^D$

\begin{theorem}
We have an equality
\begin{equation}
    H^D[\overline{\chi}^A[X+Y]]\overline{\chi}^B=-1+\dfrac{x}{2}\left( (\overline{\chi}^A[X+Y]^{\even})^2+p_2[\overline{\chi}^A[X+Y]^{\even}]\right)
\end{equation}
\end{theorem}

\begin{proof}
It is enough to recall that $s_{(2,\emptyset)}=(p_1[X]^2+p_2[X])/2$.    
\end{proof}

\section*{Acknowledgement}
The authors are members of the INdAM research group GNSAGA.
The authors are partially supported by the Italian Ministry of University and Research (MUR) under the grant FIS 2, project “Chow Ring of Toric Arrangements” (FIS-2023-00808, CUP J53C25001900001).
The first author is partially supported by  the PRIN Project 2017YRA3LK "Moduli and Lie theory". The second author was partially supported by  INdAM research fellowship.

\bibliographystyle{amsalpha}
\bibliography{references}

\end{document}